\documentclass[a4paper,11pt]{amsart}
\usepackage{amssymb,mathtools}
\usepackage[T1]{fontenc}
\usepackage{lmodern}
\usepackage{textcomp}
\usepackage{microtype}
\usepackage{enumerate}
\usepackage[english]{babel}
\usepackage{hyperref}
\newtheorem{theorem}{Theorem}[section]
\newtheorem{lemma}[theorem]{Lemma}
\newtheorem{proposition}[theorem]{Proposition}
\newtheorem{corollary}[theorem]{Corollary}
\newtheorem{model}{Model}
\theoremstyle{definition}
\newtheorem{definition}[theorem]{Definition}
\theoremstyle{remark}
\newtheorem{remark}[theorem]{Remark}

\newenvironment{mmodel}[1]{%
  \renewcommand\themodel{#1}%
  \model
}{\endmodel}
\numberwithin{equation}{section}
\newcommand{\e}{\mathrm{e}}

\newcommand{\unif}{\mathrm{unif}}
\newcommand{\RR}{\mathbb{R}}
\newcommand{\EE}{\mathbb{E}}

\newcommand{\CC}{\mathbb{C}}
\newcommand{\NN}{\mathbb{N}}
\newcommand{\ZZ}{\mathbb{Z}}

\newcommand{\cD}{\mathcal{D}}
\newcommand{\cH}{\mathcal{H}}

\newcommand{\cJ}{\mathcal{J}}
\newcommand{\eps}{\varepsilon}
\renewcommand{\theta}{\vartheta}
\renewcommand{\phi}{\varphi}
\newcommand{\indic}{\mathbf{1}}

\newcommand{\fh}{\mathfrak{h}}
\DeclareMathOperator{\Tr}{Tr}
\DeclareMathOperator{\Ran}{Ran}
\DeclareMathOperator{\Id}{Id}

\DeclareMathOperator*{\supp}{supp}
\DeclareMathOperator*{\Sym}{Sym}

\newcommand*\Diff[1]{\mathop{}\!\mathrm{d}#1}
\renewcommand{\restriction}{|}
\newcommand{\norm}[1]{\ensuremath{\left\|#1\right\|}}

\newcommand{\abs}[1]{\ensuremath{\left| #1\right|}}
\newcommand{\bes}{\begin{equation*}}
\newcommand{\ees}{\end{equation*}}
\newcommand{\be}{\begin{equation}}
\newcommand{\ee}{\end{equation}}
\newcommand{\eqs}[1]{\begin{align*}#1\end{align*}}
\newcommand{\eq}[1]{\begin{align}#1\end{align}}
\renewcommand{\div}{\mathrm{div}}
\newcommand{\grad}{\nabla}
\renewcommand{\L}{\Lambda}
\newcommand{\sfUCP}{\mathrm{sfUCP}}
\newcommand{\Ellip}{\mathrm{Ellip}}
\newcommand{\Lipschitz}{\mathrm{Lip}}
\newcommand{\evl}{\mathrm{evl}}
\makeatletter
\def\blfootnote{\gdef\@thefnmark{}\@footnotetext}
\makeatother

\title[Unique continuation for the gradient]
{Unique continuation for the gradient of eigenfunctions, and Wegner estimates for random divergence-type operators}
\hypersetup{
	pdftitle={Unique continuation for the gradient of eigenfunctions, and Wegner estimates for random divergence-type operators},
	pdfauthor={Alexander Dicke, Ivan Veseli\'c},
	hidelinks
}
\author[A.~Dicke]{Alexander Dicke}
\address[A.D.]{
	Dortmund,
	Germany
}
\email{adicke.math@gmail.com}
\author[I.~Veseli\'c]{Ivan Veseli\'c}
\address[I.V.]{
	Technische Univer\-si\-t\"at Dortmund, 
	Germany
}
\email{ivan.veselic@mathematik.tu-dortmund.de}
\urladdr{\url{https://www.mathematik.tu-dortmund.de/lsix/research/analysis/}}
\begin{document}
\maketitle
%
%
\begin{abstract}
	We prove a scale-free quantitative unique continuation estimate for the gradient of eigenfunctions of divergence-type operators, i.e.,
	operators of the form $-\div A\grad$, where the matrix function $A$ is uniformly elliptic.
	The proof uses a unique continuation principle for elliptic second-order operators and a lower bound on the $L^2$-norm of the gradient
	of eigenfunctions corresponding to strictly positive eigenvalues.

	As an application, we prove an eigenvalue lifting estimate that allows us to prove a Wegner estimate for random divergence-type operators.
	Here our approach allows us to get rid of a restrictive covering condition that was essential in previous proofs of Wegner estimates for such models.
\end{abstract}
%
%
\section{Introduction} \label{sec:notation}

	The analysis of divergence-type operators is motivated, among others, 
	by the study of propagation of electromagnetic and classical waves in media, 
	including random  ones.

	Since such operators are elliptic second-order operators they obey unique 
	continuation estimates. 
	In fact, it was recently shown that they even satisfy so-called 
	\emph{scale-free unique continuation estimates}, that have  first 
	been established for 
	Schr\"odinger operators in, e.g., \cite{CombesHK-03,GerminetK-13,RojasMolinaV-13,Klein-13} and references cited therein.
	These estimates compare the $L^2$-norm 	of an eigenfunction $\psi$ 
	(or a function in the range of an appropriate spectral projector
	of the operator under consideration) on the full domain with its $L^2$-norm on a collection of small balls that are evenly distributed throughout the domain.
	For elliptic second-order operators, analogous, but somewhat weaker, bounds were proven in \cite{BorisovTV-17,TautenhahnV-20}.
	The methods used there rely on those developed for classical (i.e., local) unique continuation estimates for elliptic second-order differential operators,
	see, e.g., \cite{JerisonK-85,KenigRS-87,Kukavica-98,EscauriazaV-03,NakicRT-19} and the literature cited therein.

	An important application of scale-free unique continuation estimates is the theory of random operators,
	where unique continuation principles are used
	to prove for instance Wegner and initial length scale estimates, see \cite{CombesHK-03,BourgainK-05,RojasMolinaV-13,Klein-13,NakicTTV-18,TautenhahnV-20,SeelmannT-20}.
	However, in all these references and in most of the existing literature, the random part of the operator is assumed to be the zeroth order term.
	In other words, the randomness is introduced by adding a random potential.

	In this paper we consider more challenging operators where the leading order term is random. 
	This situation was studied in \cite{FigotinK-96} and \cite{Stollmann-98} as a model for propagation of waves in random media, see also \cite{FigotinK-97c}.
	These papers provide a Wegner estimate, assuming however that the random perturbations satisfy a covering condition.

	Our proof demonstrates how to remove this covering condition assumed in \cite{FigotinK-96,Stollmann-98} using a scale-free unique continuation estimate
	\emph{for the gradient of eigenfunctions}.
	In contrast to usual scale-free unique continuation estimates, we compare the $L^2$-norm of the gradient of an eigenfunction on the union of balls 
	described above with the $L^2$-norm of the eigenfunction on the full domain.
	To the best of our knowledge, previously only qualitative unique continuation for the gradient has been studied, see \cite{Nkashama-10}.
	We use ideas of the latter paper and combine it with a unique continuation estimate of \cite{TautenhahnV-20}
	to obtain the desired unique continuation estimate for the gradient.

	The energy zero is not a fluctuation boundary of random divergence-type operators.
	This is illustrated by the fact that if one restricts the operator to a cube and imposes Neumann boundary conditions,
	zero is an eigenvalue regardless of the random configuration.
	Therefore, we will not only exclude high energies from our consideration, but also energies close to zero.
	Consequently, our unique continuation estimate for the gradient is only valid for eigenfunctions corresponding to strictly positive eigenvalues.

	The structure of the paper is as follows: In the next Section~\ref{sec:main-results} we introduce the notation and formulate the main results concerning 
	unique continuation for the gradient, the proof of which is postponed to Section~\ref{sec:proof-ucpg}.
	Thereafter, in Section~\ref{sec:applications}, we consider applications of our unique continuation estimate for the gradient to random divergence-type operators.
	Section~\ref{sec:small-energies} is dedicated to stronger bounds valid for small energies. 
	Some of these are based on a remark made in \cite{TautenhahnV-20} which allows us to partly remove some assumptions of our main result.
	The proof of the latter remark is postponed to Appendix~\ref{sec:appendix}.
	Finally, in Section~\ref{sec:scaling} we present a scaled version of our main result, which
	was successfully applied in \cite{Dicke-21} and might be of interest in future research.
	(Note that albeit \cite{Dicke-21} was published before the present paper, it is actually a sequel work.)
	
	Let us emphasize that, having in mind future applications in the theory of Anderson localization 
	for divergence-type operators, we formulate a number
	of similar results displaying the explicit dependence of the constants on the model parameters. This is necessary because proofs of localization depend on a delicate interplay of a number or parameters.

\section{Notation and the main result} \label{sec:main-results}

	Let $d\in\NN$ and let $\L_L=(-L/2,L/2)^d$ denote the cube with side length $L\in\NN_\infty:=\NN\cup\{\infty\}$, i.e., $\L_\infty=\RR^d$.
	Let $B(x,r)$ denote the ball with center $x\in\RR^d$ and radius $r\geq 0$ and let $A\colon \L_L\to\RR^{d\times d}$ be a matrix function such that $A(x)$ 
	is symmetric for all $x\in\L_L$ and there are constants $\theta_{\Ellip,-},\theta_{\Ellip,+}>0$ such that
	\be\label{eq:elliptic}\tag{Ellip}
		\theta_{\Ellip,-}|\xi|^2
		\leq 
		\xi\cdot A(x)\xi 
		\leq \theta_{\Ellip,+}|\xi|^2
	\ee
	for all $x\in\L_L$ and all $\xi\in\RR^d$.
	We abbreviate $\theta_\Ellip:=\max\{\theta_{\Ellip,-}^{-1},\theta_{\Ellip,+}\}$.
	For $\cD(\fh^L)=H^1_0(\L_L)\subset H^1(\L_L)$, consider the form $\fh^L\colon \cD(\fh)\times\cD(\fh)\to\CC$ defined by
	\be
		\fh^L(u,v) := \int_{\L_L}\overline{\grad u}\cdot A\grad v,
		\label{eq:eigenvalue}
	\ee
	where $\grad$ denotes the weak gradient.
	The form $\fh^L$ is densely defined, closed, symmetric, and sectorial.
	Thus, there exists a unique selfadjoint operator $H^L(A)$ 
	associated with the form $\fh^L$.
	Let us emphasize that in general the operator domain $\cD(H^L(A))$ does not contain smooth functions which
	is the reason why we rely on the form approach.
	However, if the matrix function $A$ is Lipschitz continuous, i.e., if there is a constant 
	$\theta_\Lipschitz>0$ such that
	\be\tag{Lip}\label{eq:Lipschitz}
		\norm{A(x)-A(y)}_\infty\leq \theta_\Lipschitz\abs{x-y}
	\ee
	for all $x,y\in\L_L$, 
	we have $ C_c^\infty(\L_L)\subset\cD(H^L(A))$ and on $C_c^\infty(\L_L)$ the operator $H^L(A)$ coincides with the operator
	\[
		\cH_L\colon C_c^\infty(\L_L) \to L^2(\L_L), \quad u\mapsto -\div A\grad u.
	\]
	The latter illustrates that the operator $H^L(A)$ defined above is a realization of the divergence-type operator $-\div A\grad$ on the cube $\L_L$ and 
	due to the choice $\cD(\fh^L)=H^1_0(\L_L)$ the operator $H^L(A)$ has Dirichlet boundary conditions. 
	While the main body of the paper is devoted to Dirichlet boundary conditions, we treat at several instances Neumann boundary conditions 
	(at least for energies close to zero), see Subsection~\ref{ss:Neumann} below. 
	Recall that in this case the form domain is given by $\cD(\fh^L) = H^1(\L_L)$. 

	The notion of a scale-free unique continuation principle relies on the following definition used implicitly or explicitly in the literature on
	random operators, see, for instance, \cite{RojasMolinaV-13,Klein-13, NakicTTV-18, NakicTTV-20a, TautenhahnV-20}.
	
	\begin{definition} \label{def:equidistributed-seq}
		Let $G>0$ and $\delta\in (0,G/2)$.
		A sequence $Z=(z_j)_{j\in (G\ZZ)^d}\subset \RR^d$ is said to be 
		$(G,\delta)$-equidistributed if $B(z_j,\delta)\subset \L_G(j)$ for all $j\in \ZZ^d$.
		For $L\in \NN$ we set
		\[
			S_{Z,\delta}(\infty) := \bigcup_{j\in (G\ZZ)^d} B(z_j,\delta)
			\quad\text{and}\quad S_{Z,\delta}(GL):=S_{Z,\delta}(\infty)\cap\L_{GL}.
		\]
	\end{definition}

	Initially, we only consider the case of $(1,\delta)$-equidistributed sequences.
	The general case follows from this by a scaling argument, 
	see Section~\ref{sec:scaling} below.
	In order to formulate our main result, we need to introduce a technical assumption from \cite{TautenhahnV-20}:
	Given $L\in\NN$ and $A=(a_{j,k})_{j,k=1,\dots,d}$, we say that $A$ satisfies assumption \eqref{eq:DIR-assumption} if 
	\be
		\forall k,j\in\{1,\dots, d\},k\neq j\forall 
		x\in\overline{\L_L}\cap\overline{(\L_L+L e_k)}\colon a_{j,k}(x)=a_{k,j}(x)=0.
		\tag{\text{Dir}} \label{eq:DIR-assumption}
	\ee
	Let us emphasize that this assumption is in particular satisfied if 
	all off-diagonal coefficients of $A$ vanish on the boundary $\partial \L_L$ of the cube $\L_L$.

	With this notation at hand our first main result reads as follows.

	\begin{theorem} \label{thm:sfUCPG}
		Let $L\in\NN_\infty$. 
		Assume that $A$ satisfies \eqref{eq:elliptic}, \eqref{eq:Lipschitz} and \eqref{eq:DIR-assumption}, 
		let $0<E_-<E_+<\infty$, and let $\delta_0$ 
		be sufficiently small (depending only on $d,\theta_\Ellip$ and $\theta_\Lipschitz$).
		Then for all $\delta\in (0,\delta_0)$, there exists a constant $C_\sfUCP^\grad>0$ depending only on 
		$d,\theta_{\Ellip,-},\theta_{\Ellip,+},\theta_\Lipschitz,E_-,E_+$ and $\delta$ such that for all
		$E\in (E_-,E_+)$, all $\psi\in\cD(H^L(A))$ satisfying $H^L(A)\psi=E\psi$, and all $(1,\delta)$-equidistributed sequences $Z$ we have
		\be
			\norm{\grad \psi}_{L^2(S_{Z,\delta}(L))}^2 \geq C_{\sfUCP}^\grad\norm{\psi}_{L^2(\L_L)}^2.
			\label{eq:sfUCPG}
		\ee
		The constant $C_\sfUCP^\grad$ is given in \eqref{eq:sfUCPG-const} below.
	\end{theorem}

	\begin{remark}
		One can choose $\delta_0 = 2\bigl(330d\e^2\theta_\Ellip^{11/2}(\theta_\Ellip+1)^{5/3}(\theta_\Lipschitz+1)\bigr)^{-1}$ in the previous theorem, cf.~\cite{TautenhahnV-20}.
		Moreover, an inequality like \eqref{eq:sfUCPG} is valid for $\delta\geq\delta_0$ 
		as well but with a different constant, cf.~Remark~2.4 in 
		\cite{TautenhahnV-20} and Remark~\ref{rem:delta0-assumption} below.
	\end{remark}
	
	\begin{remark}
		The theorem fails for $E_-=0$. 
		More precisely, it is possible to construct a sequence of normalized eigenfunctions $\psi_L$ corresponding to eigenvalues converging to $0$ as $L$ 
		increases such that
		\[
			\lim_{L\to\infty} \norm{\grad\psi_L}_{L^2(\L_L)}=0.
		\]
		Thus, \eqref{eq:sfUCPG} must fail.
	\end{remark}

	The stated theorem contains two assumptions which one can hope to eliminate eventually: 
	Assumption \eqref{eq:DIR-assumption} is needed for a certain extension argument used in \cite{TautenhahnV-20}. 
	It is quite possible that this step could be replaced by a generalization of an extension (possibly by a smoothing procedure) which does not require 
	the assumption \eqref{eq:DIR-assumption}.
	Besides assumption \eqref{eq:DIR-assumption}, the Lipschitz continuity of $A$ needed in Theorem~\ref{thm:sfUCPG} is a drawback for the application we have in mind. 
	However, it is possible to use an approximation argument to allow for discontinuous coefficient matrices $A$ in certain situations.
	Furthermore, in the small energy regime the Lipschitz continuity and assumption \eqref{eq:DIR-assumption}
	are not needed, see the discussion in Section~\ref{sec:small-energies} below.
	
\section{Proof of Theorem~\ref{thm:sfUCPG}} \label{sec:proof-ucpg}

	We first prove a lemma that establishes a relation between an eigenfunction and its gradient. 
	The proof is inspired by the arguments in \cite{Nkashama-10}, where the author proves qualitative unique continuation for the gradient of eigenfunctions of 
	second-order elliptic operators.
	However, this requires a strict condition on the sign of the zeroth order term. 
	In our context, this condition is partly replaced by the assumptions for the energy interval.

	\begin{lemma} \label{lem:gradient}
		Let $0<E<\infty$, let $L,r>0$, and suppose that $A$ fulfills \eqref{eq:elliptic}.
		Then there is a constant $C^\grad(r) > 0$ such that 
		for all points $x_0\in\L_L$ with $B(x_0,2r)\subset\L_L$,
		and all solutions $\psi\in\cD(H^L(A))$
		of $H^L(A)\psi=E\psi$ we have
		\be\label{eq:gradient-easy}
			\norm{\grad \psi}_{L^2(B(x_0,2r))}^2
			\geq
			C^\grad(r)
			\norm{\psi}_{L^2(B(x_0,r))}^2
			.
		\ee
		The constant is given in \eqref{eq:gradient-const} below. 
	\end{lemma}

	\begin{proof}
		Let $\phi\colon \Lambda_L\to [0,1]$ be a smooth cutoff-function satisfying
		$\phi\equiv 1$ on $B(x_0,r)$,
		$\phi\equiv 0$ on $\Lambda_L\setminus B(x_0,2r)$,
		and $\left\| \grad\phi \right\|_\infty\leq 2/r$.
		By the definition of the operator $H^L(A)$ we have $H^L(A)\psi=E\psi$
		if and only if $\psi\in H^1_0(\L_L)$ and
		$\fh^L(v,\psi) = \langle v,E\psi\rangle_{L^2(\Lambda_L)}$
		for all $v\in H^1_0(\L_L)$.
		Using the last identity with $v = \psi\phi^2 \in H^1_0(\L_L)$ we get
		\be
			\bigl\langle \psi\phi^2,E\psi\bigr\rangle_{L^2(\Lambda_L)}
			\leq \int_{\L_L}\phi^2\left| \overline{\grad \psi}\cdot A\grad \psi \right|+2\phi \abs{\psi}\left| \grad\phi\cdot A\grad \psi \right|.
			\label{eq:gradient-form}
		\ee
		Using the ellipticity of $A$ as well as Cauchy-Schwartz and Young's inequality, the right hand side is bounded from above by
		\[
			\frac{E}{2} \bigl\langle\psi\phi^2,\psi \bigr\rangle_{L^2(\Lambda_L)}
			+\int_{\L_L}\Bigl(\theta_{\Ellip,+}\phi^2
			+\frac{2\theta_{\Ellip,+}^2}{E}\abs{\grad\phi}^2\Bigr)\abs{\grad \psi}^2
			.
		\]
		Hence, \eqref{eq:gradient-form} and the bound on the gradient of $\phi$ imply
		\[
			\Bigl(E-\frac{E}{2}\Bigr)\int_{\L_L}\phi^2 \abs{\psi}^2
			\leq \Bigl( \theta_{\Ellip,+} + \frac{2\theta_{\Ellip,+}^2}{E}
			\cdot\frac{4}{r^2} \Bigr) \int_{B(x_0,2r)}\abs{\grad \psi}^2.
		\]
		Since $\phi\geq \indic_{B(x_0,r)}$, this shows
		\[
			\frac{E}{2}\norm{\psi}_{L^2(B(x_0,r))}^2
			\leq
			\Bigl( \theta_{\Ellip,+} + \frac{8\theta_{\Ellip,+}^2}{r^2E} \Bigr)\norm{\grad \psi}_{L^2(B(x_0,2r))}^2
			,
		\]
		which is equivalent to inequality \eqref{eq:gradient-easy} with 
		\be
			\label{eq:gradient-const}
			C^\grad(r) = \frac{1}{2} \frac{r^2 E^2}{\theta_{\Ellip,+}r^2E+8\theta_{\Ellip,+}^2}
			.
			\qedhere
		\ee
	\end{proof}

	\begin{remark} \label{rem:gradient}
		\begin{enumerate}[(i)]
			\item The lemma gives in a sense a reverse Cacciopoli-type inequality valid under 
			certain conditions.
			\item The constant $C^\grad(r)$ is independent from the point $x_0$ 
			and only the ellipticity of $A$ was used.
			\item The proof applies to all operators that correspond to forms 
			$\fh^L$ as above with form domain $\cD(\fh^L)$
			satisfying $\psi\phi^2 \in \cD(\fh^L)$ for all cut-off functions $\phi$ 
			and all $\psi\in\cD(\fh^L)$.
		\end{enumerate} 
	\end{remark}

	The proof of Theorem~\ref{thm:sfUCPG} relies on the main result of \cite{TautenhahnV-20}.
	For the sake of completeness, we here recap a version of it suited to our purpose.

	\begin{theorem}[see \cite{TautenhahnV-20}] \label{thm:sfUCP-TV19}
		Assume that $A$ fulfills \eqref{eq:elliptic}, \eqref{eq:Lipschitz} and \eqref{eq:DIR-assumption}. 
		Then there exists a constant $N>0$ depending only on $d$, $\theta_{\Ellip}$ and $\theta_\Lipschitz$ such that for all $L\in\NN_\infty$, all 
		measurable and bounded $V\colon \L_L\to\RR$, all $\delta\in (0,\delta_0/2]$, 
		all $(1,\delta)$-equidistributed sequences $Z$, and all 
		$\psi\in\cD(H^L(A))$ satisfying $|H^L(A)\psi|\leq |V\psi|$ almost everywhere on $\L_L$, we have
		\be
			\norm{\psi}_{L^2(S_{Z,\delta}(L))}^2 \geq C_{\sfUCP}\norm{\psi}_{L^2(\Lambda_L)}^2.
			\label{eq:sfUCP-TV19}
		\ee
		Here
		\be
			C_{\sfUCP}=C_{\sfUCP}(\delta)=\delta^{N(1+\norm{V}_\infty^{2/3})}.
			\label{eq:sfUCP-TV19-const}
		\ee
	\end{theorem}

	Combining this theorem with Lemma~\ref{lem:gradient} we prove Theorem~\ref{thm:sfUCPG}.
	
	\begin{proof}[Proof of Theorem~\ref{thm:sfUCPG}]
		The sequence $Z$ is clearly $(1,\delta/2)$-equi\-distri\-buted. 
		We apply Lemma~\ref{lem:gradient} followed by Theorem~\ref{thm:sfUCP-TV19} with $V\equiv E_+$ on $\L_L$ to obtain
		\be
			\norm{\grad \psi}_{L^2(S_{Z,\delta}(L))}^2 \geq C^{\grad}(\delta) \norm{\psi}_{L^2(S_{Z,\delta/2}(L))}^2
			\geq C^\grad_\sfUCP \norm{\psi}_{L^2(\L_L)}^2
			,
		\ee
		where the constant is given by
		\be
			C^\grad_{\sfUCP}=C^\grad_{\sfUCP}(\delta)
			=\frac{\delta^2E_-^2}{2\theta_{\Ellip,+}
			\bigl(8\theta_{\Ellip,+}+\delta^2E_-\bigr)}
			\Bigl(\frac{\delta}{2}\Bigr)^{N(1+E_+^{2/3})}.
			\label{eq:sfUCPG-const}\qedhere
		\ee
	\end{proof}

	\begin{remark}\label{rem:delta0-assumption}
		\begin{enumerate}[(i)]
			\item We can allow arbitrary $\delta\in (0,1/2]$ if we modify the definition of the constant in
			\eqref{eq:sfUCP-TV19-const} to
			\[
				C_{\sfUCP}=C_{\sfUCP}(\delta) = \bigl(\min\{\delta,\delta_0\}\bigr)^{N(1+\norm{V}_\infty^{2/3})}
			\]
			and similarly substitute $\min\{\delta,\delta_0\}$ for $\delta$ in the constant $C^\grad_{\sfUCP}$ 
			of Theorem~\ref{thm:sfUCPG} given in \eqref{eq:sfUCPG-const} above.
			\item Since $\delta\leq 1$, we see that
			\be
				C_1\Bigl(\frac{\delta}{2}\Bigr)^{2+N(1+E_+^{2/3})}
				\leq C^\grad_{\sfUCP}(\delta) 
				\leq C_2\Bigl(\frac{\delta}{2}\Bigr)^{N(1+E_+^{2/3})}
			\ee
			for constants $C_1,C_2$ depending only on $E_-$ and $\theta_{\Ellip,+}$.
		\end{enumerate}
	\end{remark}

\section{Applications} \label{sec:applications}

	Here we apply our main results in the theory of random operators. 
	As noted in the introduction, we want to understand the case where the second order term is random. 
	To this end, we need to understand how a modification of the coefficient matrix affects the eigenvalues.
	In particular, we investigate the movement of eigenvalues if we perturb the matrix function $A$ by some non-negative function  $W$ times the  identity matrix.
	The novelty is that we are able to treat the case that $W$ has, in some sense, small support. 
	In a second step, we use this result to prove a Wegner estimate for random divergence-type operators.

	Throughout this section we assume that $L\in\NN$, i.e., the geometric domain is a finite cube.
	Let us point out that in this situation the divergence-type operators 
	under consideration have compact resolvent and therefore purely discrete spectrum.

\subsection{Eigenvalue lifting} \label{sec:Eigenvalue-lifting}

	Given $\delta\in(0,1/2)$ and a $(1,\delta)$-equidistributed sequence $Z$, we want to investigate how the eigenvalues of the operator 
	$H_t^L:=H^L(A+t\, W\Id)$ vary as $t$ increases.
	Here $A$ is a matrix function as defined above that at least satisfies \eqref{eq:elliptic} and $W$ satisfies $W\geq\indic_{S_{Z,\delta}(L)}$.
	Fix $T>0$ and recall that for $t\in [0,T]$ the operator $H_t^L$ is associated to the form
	\be
		\fh_t^L\colon \cD(\fh_t^L)\times \cD(\fh_t^L)\to\CC, 
		\quad 
		(u,v)\mapsto \int_{\L_L}\overline{\grad u}\cdot (A+t\,W\Id)\grad v
		\label{eq:evl-form}
	\ee
	with $\cD(\fh_t^L)=H^1_0(\L_L)$.
	It is not hard to see that there is a domain $D\subset\CC$ such that $[0,T]\subset D$ and such that $\fh_t^L$ is densely defined and sectorial for all $t\in D$.
	Hence, the family of forms $(\fh_t^L)_{t\in D}$ turns out to be  a holomorphic family of type (a) in the sense of Kato \cite{Kato-80}. 
	Moreover, the quadratic form is increasing in $t$, i.e.,
	\be
		\fh_t^L(u,u)\leq \fh^L_s(u,u)
		\label{eq:evl-form-increasing}
	\ee
	for $t\leq s$ and we easily calculate
	\be
		\Bigl(\frac{\Diff}{\Diff{z}}\fh_{z}^L\Bigr)(u,u) 
		:= 
		\lim_{w\to z} \frac{\fh_{w}^L(u,u)-\fh_{z}^L(u,u)}{w-z} = \int_{\L_L}W|\grad u|^2.
		\label{eq:evl-form-derivative}
	\ee

	Thus, $(H_t^L)_{t\in D}$ is a holomorphic family of type (B) in the sense of Kato and $H^L_t$ is selfadjoint, lower-semibounded and has compact resolvent for $t\in [0,T]$.
	We denote by $(E_n^L(t))_{n\in\NN}$ the eigenvalues of $H_t^L$, enumerated non-decreasingly and counting multiplicities.

	In order to exploit our unique continuation estimate for the gradient, 
	we first assume that $W$ and $A$ are Lipschitz continuous and prove the following result. 

	\begin{theorem} \label{thm:evl}
		Let $T, \kappa, K_1, K_2> 0$, $\delta\in (0,1/2]$, $0<E_-<E_+<\infty$, and assume that $A$ satisfies \eqref{eq:elliptic}, \eqref{eq:Lipschitz} and \eqref{eq:DIR-assumption}.
		Then there is a constant $C_\evl>0$ such that for all $(1,\delta)$-equidistributed sequences $Z$, all Lipschitz continuous $W\in L^\infty(\L_L)$
		satisfying $\Lipschitz(W)\leq K_1$, $\norm{W}_\infty\leq K_2$ and $W\geq \kappa \indic_{S_{Z,\delta}(L)}$, 
		and all $n\in\NN$ such that $E_-<E_n^L(0)\leq E_n^L(T)<E_+$,
		we have
		\be
			E_n^L(t) \geq E_n^L(0)+\kappa t\,C_\evl.
			\label{eq:evl}
		\ee
		The constant $C_\evl$ is given in \eqref{eq:evl-const} below.
	\end{theorem}

	\begin{proof}
		There are two sequences $(\lambda_\ell)_{\ell\in\NN}$ and 
		$(\phi_\ell)_{\ell\in\NN}$ of analytic functions on $D\cap\RR$
		such that $H_t^L\phi_\ell(t)=\lambda_\ell(t)\phi_\ell(t)$ for all $t\in [0,T]$, 
		see \cite[VII. Remark 4.22 and VII. Theorem 3.9]{Kato-80}. 
		Here, the $(\lambda_\ell(t))_\ell$ are all the repeated eigenvalues of $H_t^L$ with corresponding normalized eigenfunctions $(\phi_\ell(t))_\ell$.
		The functions $E_n^L$ are formed by connecting several $\lambda_\ell$ in a continuous manner, i.e., $E_n^L$ may jump from one $\lambda_\ell$ 
		to another at every crossing point between the different $\lambda_\ell$.
		Since the $\lambda_\ell$ are analytic and the $H_t^L$ are lower-semibounded 
		there are at most finitely many points where these jumps occur.

		More precisely, there are $M\in\NN$, $\ell_1,\dots,\ell_M\in\NN$, and $0=t_1<\dots<t_{M+1}=T$ such that
		\be
			E_n^L(t)=\lambda_{\ell_j}(t) \quad\text{for all}\quad t\in [t_j,t_{j+1}],
			\label{eq:evl-connecting-eigenvalues}
		\ee
		where the overlap is continuous. 
		Especially, the function $E_n^L$ agrees piecewise with some $\lambda_\ell$ and is therefore piecewise analytic. 
		The argument in \cite[proof of VII.Theorem 4.21]{Kato-80} shows
		that
		\be
			\Bigl(\frac{\Diff}{\Diff{t}} E_n^L\Bigr)(t) 
			= \Bigl(\frac{\Diff}{\Diff{t}} \fh_t^L\Bigr)\bigl(\phi_\ell(t),\phi_\ell(t)\bigr) 
			\label{eq:eigenvalue-derivative}
		\ee
		if $E_n^L(\cdot)=\lambda_\ell(\cdot)$ in a neighborhood of $t$. 
		Only at the points $t\in\{t_1,\ldots, t_{N+1}\}$, at which a crossing from one $\lambda_\ell$ to another occurs, it is possible that $E_n^L(\cdot)$ 
		is not differentiable.
		Hence, equality \eqref{eq:eigenvalue-derivative} holds for all but finitely many $t\in [0,T]$.

		Using identity \eqref{eq:evl-form-derivative}, we calculate that 
		the right hand side of \eqref{eq:eigenvalue-derivative} satisfies
		\be
			\Bigl(\frac{\Diff}{\Diff{t}} \fh_t^L\Bigr)\bigl(\phi_\ell(t),\phi_\ell(t)\bigr) 
			= 
			\int_{\L_L}W|\grad\phi_\ell(t)|^2
			\geq \kappa \norm{\grad\phi_\ell(t)}_{L^2(S_{Z,\delta}(L))}^2.
			\label{eq:lifting}
		\ee
		We want to apply the unique continuation estimate for the gradient on the right hand side of \eqref{eq:lifting}. 
		To this end, we need to verify the assumptions of Theorem~\ref{thm:sfUCPG}:
		For fixed $t$, the matrix function $x\mapsto A(x)+tW(x)\Id$ is Lipschitz continuous with Lipschitz-constant 
		$\tilde{\theta}_\Lipschitz(t)=\theta_\Lipschitz+tK_1$ and elliptic with ellipticity constants given 
		by $\tilde{\theta}_{\Ellip,+}(t)=\theta_{\Ellip,-}$ and 
		$\tilde{\theta}_{\Ellip,+}(t)= \theta_{\Ellip,+}+tK_2$.
		Moreover, the assumption on $E_n^L(\cdot)$ shows that $\phi_\ell(t)$ is a normalized eigenfunction 
		whose eigenvalue lies between $E_-$ and $E_+$.
		
		Applying Theorem~\ref{thm:sfUCPG} for all $t\in [0,T]$ provides us with a constant
		$C_\sfUCP^\grad=C_\sfUCP^\grad(\delta,d,\tilde{\theta}_{\Ellip,\pm}(t),\tilde{\theta}_\Lipschitz(t))$ such that
		\[
			\norm{\grad\phi_\ell(t)}_{L^2(S_{Z,\delta}(L))}^2 \geq C_\sfUCP^\grad\norm{\phi_\ell(t)}_{L^2(\L_L)}^2 =   C_\sfUCP^\grad.
		\]
		Thus,
		\[
			\norm{\grad\phi_\ell(t)}_{L^2(S_{Z,\delta}(L))}^2 \geq C_\evl:=  \inf_{t\in [0,T]}C_\sfUCP^\grad
		\]
		where $C_\evl$ does not depend on $t$.
		Indeed,
		\be
			C_\evl=\frac{\delta^2E_-^2}{2\tilde{\theta}_{\Ellip,+}(T)
			\bigl(8\tilde{\theta}_{\Ellip,+}(T)+\delta^2E_-\bigr)}
			\Bigl(\frac{\delta}{2}\Bigr)^{N(1+E_+^{2/3})},
			\label{eq:evl-const}
		\ee
		where $N=N(\tilde{\theta}_\Lipschitz(T),\tilde{\theta}_{\Ellip,\pm}(T))$. 
		Consequently,
		\be
			\Bigl(\frac{\Diff}{\Diff{t}} E_n^L\Bigr)(t) \geq \kappa C_\evl \label{eq:lower-bound}
		\ee
		for all $t \in (0,T)\setminus \{t_2,\ldots, t_M\}$. 
		Finally, for fixed $t\in [0,T]$, we let $\tilde{M}\in\{1,\dots,M\}$ be the smallest 
		number such that $t_n>t$ for all $n> \tilde{M}$ and obtain
		\eqs{
			E_n^L(t) &
			 = E_n^L(0)+\sum_{j=1}^{\tilde{M}-1}\int_{t_j}^{t_{j+1}} \frac{\Diff}{\Diff{s}} E_n^L(s) \Diff{s}
			 +\int_{t_{\tilde{M}}}^t\frac{\Diff}{\Diff{s}} E_n^L(s) \Diff{s}\\
			 &\geq  E_n^L(0)+\kappa t\,C_\evl.\qedhere
		}
	\end{proof}

	\begin{remark}
		Let us stress the difference of our situation to the one we encounter for Schr\"odinger operators with alloy-type potentials:
		To that end, let us define $\tilde{H}^L_t=-\Delta+t\,W$ with its associated form $\tilde{\fh}_t^L$ on $H^1_0(\L_L)$.
		Using in this case scale-free unique continuation estimates for eigenfunctions of $\tilde{H}^L_t$
		that were proven in, e.g., \cite{RojasMolinaV-13,NakicTTV-18}, we obtain
		\[
			\Bigl(\frac{\Diff}{\Diff{t}} \tilde{\fh}_t^L\Bigr)(\tilde{\phi}_\ell(t),\tilde{\phi}_\ell(t)) = \int_{\L_L}W|\tilde{\phi}_\ell(t)|^2
			\geq \norm{\tilde{\phi}_\ell(t)}_{L^2(S_{Z,\delta}(L))}^2 \geq C_\sfUCP.
		\]
		In particular, we see that in the latter case the gradient does not appear, in contrast to the present situation in \eqref{eq:lifting}.
	\end{remark}
	
	\begin{remark}
		If we assume that $W\geq 1$ on the whole cube $\L_L$, Theorem~\ref{thm:evl} is already implicitly stated
		in \cite{Stollmann-98}. 
		Unique continuation for the gradient is not needed and consequently we do not need the Lipschitz 
		continuity of the coefficients either. 
		In fact, in this particular case \eqref{eq:evl} is a simple consequence of the following elementary argument: 
		As shown in the previous proof we have
		\[
			\Bigl(\frac{\Diff}{\Diff{t}} E_n^L\Bigr)(t) \geq \norm{\grad\psi_n(t)}_{L^2(\L_L)}^2
		\]
		where $H^L_t\psi_n(t)=E_n^L(t)\psi_n(t)$. 
		Using the upper bound on the coefficient matrix and the lower bound on the eigenvalue we derive 
		\eqs{
			\langle E_n^L(t)\psi_n(t),\psi_n(t) \rangle_{L^2(\L_L)}
			& = \langle H^L_t\psi_n(t),\psi_n(t) \rangle_{L^2(\L_L)} \\
			& =\langle (A+t\,W)\grad\psi_n(t),\grad\psi_n(t) \rangle_{L^2(\L_L)} \\
			& \leq (\theta_{\Ellip,+}+T\norm{W}_\infty)\norm{\grad\psi_n(t)}_{L^2(\L_L)}^2.
		}
		Hence,
		\[
			\Bigl(\frac{\Diff}{\Diff{t}} E_n^L\Bigr)(t)
			\geq \frac{E_-}{\theta_{\Ellip,+}+T\norm{W}_\infty}.
		\]
	\end{remark}

	In case that only the matrix $A$ satisfies \eqref{eq:Lipschitz}, we may use a monotonicity argument to prove eigenvalue lifting for perturbations $W$ which 
	are not Lipschitz continuous. 
	This is the tenor of our next result. 

	\begin{corollary} \label{cor:evl-discont-perturbation}
		Let $T>0$, $\delta\in (0,1/2)$, and $0<E_-<E_+<\infty$.
		Assume that $A$ satisfies \eqref{eq:elliptic}, \eqref{eq:Lipschitz} 
		and \eqref{eq:DIR-assumption}.
		Then, there is a constant $\hat{C}_\evl>0$ such that for all 
		$(1,\delta)$-equidistributed sequences $Z$, all
		$W\in L^\infty(\L_L)$ satisfying $W\geq \indic_{S_{Z,\delta}(L)}$, 
		and all $n\in\NN$ such that $E_-<E_n^L(0)\leq E_n^L(T)<E_+$
		we have
		\be
			E_n^L(t) \geq E_n^L(0)+t\,\hat{C}_\evl.
			\label{eq:evl-discontinuous}
		\ee
	\end{corollary}

	\begin{proof}
		There exists an auxiliary Lipschitz continuous function $\tilde{W}$
		with Lipschitz-constant $\Lipschitz(\tilde{W}) = 2$ satisfying 
		$W\geq \delta \min\{1,W\} \geq \tilde{W}\geq \indic_{S_{Z,\delta/2}(L)}$.
		We define the auxiliary operator
		$\tilde{H}_t^L:=H^L(A+t\,\tilde{W}\Id)$ and denote its eigenvalues 
		enumerated non-decreasingly and counting multiplicities
		by $\tilde{E}^L_n(t)$. 
		Then we have by monotonicity
		\[
			E_-<E_n^L(0)=\tilde{E}_n^L(0)\leq\tilde{E}_n^L(T)\leq E_n^L(T)<E_+.
		\]
		Since the matrix function $x\mapsto A(x)+\,\tilde{W}(x)\Id$ has
		Lipschitz-constant at most $\theta_\Lipschitz'=\theta_\Lipschitz + 2T$ and ellipticity
		constants $\theta_{\Ellip,-}'=\theta_{\Ellip,-}$ and 
		$\theta_{\Ellip,+}' = \theta_{\Ellip,+} + T$, 
		Theorem~\ref{thm:evl} shows that
		\[
			\tilde{E}_n^L(t) \geq \tilde{E}_n^L(0) + C_\evl
		\]
		with 
		\[
			\hat{C}_\evl = 
			\frac{\delta^2E_-^2}{2(\theta_{\Ellip,+}+T)(32\theta_{\Ellip,+}+32T+\delta^2E_-)}
			\Bigl(\frac{\delta}{4}\Bigr)^{N(1+E_+^{2/3})},
		\]
		where $N$ is a constant that depends on 
		$\delta$, $T$, $\theta_{\Ellip,\pm}$ and $\theta_\Lipschitz$.
		The minimax-principle finally implies
		\[
			E_n^L(t) \geq \tilde{E}_n^L(t) \geq \tilde{E}_n^L(0) + C_\evl =  E_n^L(0) + C_\evl.
			\qedhere
		\]
	\end{proof}

\subsection{Wegner estimate}\label{subsec:Wegner}

	In this section we prove a Wegner estimate for random divergence-type operators.
	We consider a generalized alloy-type random perturbation as introduced in \cite{RojasMolinaV-13,Klein-13}.
	Let us introduce the model at hand:

	\begin{mmodel}{(A)}[Generalized alloy-type] \label{model:A}
		Let $0<\delta_-<\delta_+<\infty$, $\delta_-\in (0,1/2)$, 
		$0<C_-<C_+<\infty$, $ K >0$, and let $\omega=(\omega_j)_{j\in\ZZ^d}$ be a sequence of independent
		random variables with probability distributions $(\mu_j)_{j\in\ZZ^d}$ 
		satisfying $\supp\mu_j\subset [0,m]$ for some $m>0$. 
		Denote by $s\colon [0,\infty)\to[0,1]$ the global modulus of 
		continuity of the family $(\mu_j)_{j\in\ZZ^d}$, that is
		\be
			\sup_{j\in\ZZ^d}\{\mu_j\bigl([E-\eps/2,E+\eps/2]\bigr):E\in\RR\} \leq s(\eps)\quad\text{for}\quad\eps > 0.
			\label{eq:modulus-continuity}
		\ee
		Moreover, let $Z=(z_j)_{j\in\ZZ^d}$ be a $(1,\delta_-)$-equidistributed sequence and 
		let $(u_j)_{j\in\ZZ^d}$ be a sequence of functions on $\L_L$ satisfying
		\be
			C_-\indic_{B(z_j,\delta_-)}\leq u_j\leq C_+\indic_{B(z_j,\delta_+)}.
			\label{eq:single-site-perturbations}
		\ee
		Assume that each $u_j$ is Lipschitz continuous with Lipschitz constant $K$. 
		Lastly, we fix a matrix-function $A$ that satisfies \eqref{eq:elliptic}, 
		\eqref{eq:Lipschitz} and \eqref{eq:DIR-assumption} for all $\L_L$, $L\in\NN$.
		Then, with 
		\be
			V_\omega(x):=\sum_{j\in\ZZ^d}\omega_ju_j(x)
			\quad\text{and}\quad A_\omega:=A+V_\omega\Id
			\label{eq:crooked-Alloy-type}
		\ee
		we consider the random divergence-type operator $H^L_\omega:=H^L(A_\omega)$.
	\end{mmodel}

\begin{remark}\label{rm:uniform}
	In what follows, we refer to the parameters $\delta_\pm$, $C_\pm$, $K$, $m$, 
	$\theta_{\Lipschitz}$, $\theta_{\Ellip,\pm}$, and $d$ introduced in the
	hypothesis above as the \emph{model parameters}.
\end{remark}

A function $V_\omega$ as in \eqref{eq:crooked-Alloy-type} is clearly non-negative and bounded.
More precisely, its $L^\infty$-norm depends only on model parameters.
Hence, the random matrix function $A_\omega$ is elliptic (uniformly
with respect to $\omega$ and $x$) with ellipticity constants
\[
	\tilde{\theta}_{\Ellip,-}
	=
	\theta_-
	\quad\text{and}\quad
	\tilde{\theta}_{\Ellip,+}
	=
	\theta_+ +(2+\delta_+)^d m C_+
\]
and Lipschitz continuous with Lipschitz constant
$\tilde{\theta}_{\Lipschitz} = \theta_\Lipschitz + (2+\delta_+)^d m K$.
Hence, there is a $N_{\unif}$ depending only on $\tilde{\theta}_{\Ellip,-}, \tilde{\theta}_{\Ellip,+}, \tilde{\theta}_{\Lipschitz}$
such that \eqref{eq:evl} holds with the constant $C_\evl$ from \eqref{eq:evl-const} where $N = N_{\unif}$ for all configurations $A_\omega$.
In particular, $\norm{A_\omega}_{\Lipschitz}$ and $N_{\unif}$ are uniformly bounded by a constant depending
only on the model parameters.

\begin{remark}\label{rem:simpliefied-evl-constant}
	In fact, one can easily show that \eqref{eq:evl} holds with the simplified constant 
	\bes
		C_\evl = E_-^2\delta^{N_{\unif}(1+E_+^{2/3})}
	\ees
	when adjusting $N_{\unif}$ appropriately. 
	For convenience, we only use this constant in the rest of this section. 
\end{remark}

	We denote the eigenvalues of the operator $(H^L_\omega)_\omega$ by 
	$E_n^L(\omega)$, enumerated non-decreasingly and counting multiplicities.
	Using the results on eigenvalue lifting, we may prove a Wegner estimate for the random divergence-type operator $(H^L_\omega)_\omega$.

	\begin{remark}
		In general, the random operator $(H^L_\omega)_\omega$ is non-ergodic. 
		However, if we assume that the random variables $(\omega_j)$ are iid, 
		the single-site perturbations satisfy $u_j=u(\cdot-j)$ for some bounded,
		compactly supported, and Lipschitz continuous function $u\geq \indic_{B(0,\delta_-)}$, and the matrix $A$ is $\ZZ^d$-periodic, then the random operator is ergodic and has therefore
		almost sure spectrum.
		Let us emphasize that for $\ZZ^d$-periodic $A$ the assumption 
		\eqref{eq:DIR-assumption} is satisfied for all cubes $\L_L$ if it is satisfied for $\L_1$.
	\end{remark}
	
	\begin{theorem}[Wegner estimate] \label{thm:Wegner-general}
		Suppose that $(H^L_\omega)_\omega$ is of type~\ref{model:A} and let $0<E_-<E_+<\infty$. 
		Then there exists a constant $C_W>0$ such that for every 
		$L\in\NN$, $E>0$, and $0<\eps\leq 1/4$ satisfying $[E-7\eps,E+7\eps]\subset [E_-,E_+]$ we have
		\be
			\EE\Bigl[\Tr\chi_{[E-\eps,E+\eps]}\bigl(H_\omega^L\bigr)\Bigr] \leq C_W s(\eps)L^{2d}.
			\label{eq:Wegner-est}
		\ee
	\end{theorem}

	The proof closely follows \cite{HundertmarkKNSV-06} and uses the following partial integration 
	formula for singular distributions proved in that paper.

	\begin{lemma}[{see \cite[Lemma~6]{HundertmarkKNSV-06}}] \label{lem:PI-singular-distributions}
		Let $\mu$ be a probability measure with support contained in $(a,b)$, $a<b$, 
		and global modulus of continuity $s\colon [0,\infty)\to [0,1]$. 
		Let $\Phi$ be a continuously differentiable, non-decreasing function on $(a,b+\alpha)$. 
		Then
		\be
			\int_\RR \Phi(\lambda +\eta)-\Phi(\lambda)\Diff{\mu(\lambda)} 
			\leq s(\eta) \bigl(\Phi(b+\eta)-\Phi(a)\bigr)
			\label{eq:singular-distributions}
		\ee
		for all $\eta\in (0,\alpha)$.
	\end{lemma}

\begin{proof}[Proof of Theorem \ref{thm:Wegner-general}]
	Let $Q := \L_{L+2\delta_+ + 1}\cap\ZZ^d$ be all the
	indices $j$ such that $\omega_j$ affects the random perturbation
	$V_\omega\restriction_{\L_L}$.
	This means that the random variable $H_\omega^L$ is measurable
	w.r.t.~the $\sigma$-algebra generated by $\{\omega_j \, : \, j \in Q\}$.
	Furthermore, let $r=(r_i)_{i=1,\dots,\#Q}$ be an enumeration of the
	lattice points in $Q$.
	Define the following vectors in $\{0,1\}^{\#Q}\subset \RR^{\#Q}$:
	Let $e:=(1,\dots,1)$, let $e_\ell$ be the vector where only the
	$\ell$-th entry is $1$ while all others are zero, and let
	\[
		e_\ell^{(r)}:=\sum_{j=1}^\ell e_{r_j}
		\quad\text{for}\quad
		\ell\in\{1,\dots,\#Q\}
		.
	\]
	We set
	\[
		V^Q_{\omega+t\cdot e_\ell^{(r)}}
		:=
		V_\omega^Q+t\,\sum_{j=1}^\ell e_{r_j}u_{r_j}
		\quad\text{where}\quad
		V_\omega^Q
		:=
		\sum_{j\in Q}\omega_ju_j
		.
	\]

	Choose a monotone increasing function $\rho_\eps\in C^\infty(\RR,[-1,0])$
	satisfying $\rho\equiv -1$ on $(-\infty,-\eps]$,
	$\rho\equiv 0$ on $[\eps,\infty)$, and $\norm{\rho_\eps'}_\infty\leq 1/\eps$.
	With this smooth switch function at hand, we have
	\be
		\indic_{[E-\eps,E+\eps]}
		\leq
		\rho_\eps(\cdot-E+4\eps-2\eps)-\rho_\eps(\cdot-E-2\eps)
		\leq \indic_{[E-3\eps,E+3\eps]}
		\label{eq:smear-fct}
	\ee
	and the spectral theorem implies
	\eq{\label{eq:smear}
		\EE&
		\Bigl(\Tr\bigl[\chi_{[E-\eps,E+\eps]}(H_\omega^L)\bigr]\Bigr) \nonumber\\
		&\leq
		\EE\Bigl(\Tr\bigl[\rho_\eps(H^L_\omega-E+4\eps-2\eps)-\rho_\eps(H^L_\omega-E-2\eps)\bigr]\Bigr) \nonumber \\
		&
		=
		\EE\biggl(\sum_{n\in\NN}\Bigl[\rho_\eps(E_n^L(\omega)+4\eps-E-2\eps)-
			\rho_\eps(E_n^L(\omega)-E-2\eps) \Bigr] \biggr)
		.
	}
	It is worth pointing out that due to the upper bound in \eqref{eq:smear-fct} only those $n\in\NN$ with
	$E_n^L(\omega) \in [E-3\eps,E+3\eps]$ give a non-zero contribution in \eqref{eq:smear}.
	Set
	\begin{align}\nonumber
		C_{\mathrm{subst}}:=& E_-^2 \, \delta_-^{N_{\unif}\cdot(1+E_+^{2/3})}, \\
		\label{eq:rescaled-epsilon}
		\eta=\eta(\omega):= 
		&\max\bigl\{ t\in \bigl[0, 4\eps / C_{\mathrm{subst}}\bigr] : E_n^L(\omega+t\cdot e)\leq E_+\bigr\}
		\leq \frac{4\eps}{C_{\mathrm{subst}}},
	\end{align}
	where $N_{\unif}$ is as in Remark~\ref{rem:simpliefied-evl-constant} above. 
	This implies $E_n^L(\omega+\eta\cdot e)\leq E_+$ additionally to
	$E_-\leq E-3\eps\leq E_n^L(\omega)\leq E_n^L(\omega+\eta\cdot e)$ by the definition of $\eta$ and monotonicity.
	Note that definition \eqref{eq:rescaled-epsilon} also implies that either $\eta=4\eps / C_{\mathrm{subst}}$ or $ E_n^L(\omega+t\cdot e)= E_+$.
	In any case, the eigenvalue lifting Theorem~\ref{thm:evl} applies with $T=\eta$ so that we obtain
	\eqs{
		E_n^L(\omega+&\eta\cdot e) \\
		&\geq
		\begin{cases}
		 E_+ \geq (E+3\eps) +4 \eps,
		 & \text{if } \eta<4\eps / C_{\mathrm{subst}} \text{ in \eqref{eq:rescaled-epsilon}} \\
		 E_n^L(\omega)+\eta\, E_-^2 \, \delta_-^{N_{\unif}\cdot (1+E_+^{2/3})} ,
		 & \text{if } \eta=4\eps / C_{\mathrm{subst}} \text{ in \eqref{eq:rescaled-epsilon}}
		\end{cases}
		\\
		&\geq E_n^L(\omega) +4\eps
	}
	using the above mentioned alternative.
	We insert this bound in \eqref{eq:smear} to obtain
	\eq{\label{eq:smear-sum}
		\EE
		\Bigl(\Tr&\bigl[\chi_{[E-\eps,E+\eps]}(H_\omega^L)\bigr]\Bigr)
		\nonumber\\
		&\leq
		\EE\biggl(\sum_{n\in\NN}\Bigl[\rho_\eps(E_n^L(\omega+\eta\cdot e)-E-2\eps)-\rho_\eps(E_n^L(\omega)-E-2\eps) \Bigr] \biggr)\nonumber\\
		&=
		\EE\Bigl(\Tr\bigl[\rho_\eps(H^L_{\omega+\eta\cdot e}-E-2\eps)-\rho_\eps(H^L_\omega-E-2\eps) \bigr]\Bigr) \nonumber\\
		&=
		\sum_{\ell=1}^{\#Q} \EE\Bigl(\Tr \bigl[\rho_\eps(H^L_{\omega+\eta\cdot e_\ell^{(r)}}-E-2\eps)
		-\rho_\eps(H^L_{\omega+\eta\cdot e_{\ell-1}^{(r)}}-E-2\eps) \bigr]\Bigl)
		.
	}
	We handle each of the $\# Q$ summands in the previous telescoping sum separately.
	To this end, we fix an $\ell\in\{1,\dots,\#Q\}$ and define
	\[
		\omega^{\perp} := (\omega_k^{\perp})_{k\in Q}
		\quad\text{where}\quad
		\omega_k^{\perp} =
		\begin{cases}
			0, &\text{if}~k=r_\ell \\
			\omega_k, &\text{otherwise}
		\end{cases}
		,
	\]
	as well as
	\[
		\Phi_\ell(t)
		:=
		\Tr\bigl[\rho_\eps(H^L_{\omega^{\perp}+\eta\cdot e_{\ell-1}^{(r)}+t\cdot e_{r_\ell}}-E-2\eps)\bigr]
		\leq
		0
		.
	\]
	Thus the $\ell$-th term in the telescoping sum in \eqref{eq:smear-sum} can be rewritten as
	\eq{
		\EE&\Bigl(\Tr \Bigl[\rho_\eps(H^L_{\omega+\eta\cdot e_\ell^{(r)}}-E-2\eps)
		-\rho_\eps(H^L_{\omega+\eta\cdot e_{\ell-1}^{(r)}}-E-2\eps) \Bigr]\Bigr) \nonumber\\
		& =
		\EE^{ Q\setminus\{r_\ell\}}\Biggl[
		\int \left(\Phi_\ell(\omega_{r_\ell} +\eta) - \Phi_\ell(\omega_{r_\ell})\right)\Diff{\mu_{r_\ell}(\omega_{r_\ell})}\Biggr],
		\label{eq:expectation-telescoping-sum}
	}
	where $\EE^{ Q\setminus\{r_\ell\}}$ denotes the expectation with respect
	to all random variables $\omega_j$ with $j\in Q\setminus\{r_\ell\}$.

	In order to apply Lemma~\ref{lem:PI-singular-distributions}
	we have to show that the function $\Phi_\ell$ is continuously
	differentiable, bounded and non-decreasing on $(a,m+\alpha)$ for some $a<0,\alpha>0$.
	To begin with, the definition of $\rho_\eps$ and the minimax-principle imply that
	$\Phi_\ell$ is bounded and non-decreasing.
	In order to establish differentiability, we keep $\ell$, $\omega^\perp$, as well as $\eta$ fixed
	and consider the operator
	\[
		\tilde{H}_t^L
		=
		H^L\Bigl(A+V_{\omega^{\perp}+\eta\cdot e_{\ell-1}^{(r)} +t\cdot e_{r_\ell}}\Id\Bigr)
		.
	\]
	Recall that as discussed in Subsection~\ref{sec:Eigenvalue-lifting} above, there is a domain
	$D\subset\CC$ satisfying $D\supset [0,m]$ such that
	$(\tilde{H}_t^L)_{t\in D}$
	is a holomorphic family of type (B) in the sense of Kato \cite{Kato-80} while
	$\tilde{H}^L_t$ is selfadjoint, lower-semibounded and has compact resolvent
	for all $t \in D \cap \RR$.
	In fact, it is possible to ensure $D\supset (a,m+\alpha)$, where 
	$\alpha= C_{\mathrm{subst}}^{-1}$ and $a = -\theta_{\Ellip,-}/(2(2+\delta_+)^d m C_+)$
	is chosen in such a way that
	\[
		A + a \sum_{j\in\ZZ^d}u_j
		\geq \Big(\theta_{\Ellip,-} + a 	(2+\delta_+)^d m C_+\Big) \Id
		\geq \frac{\theta_{\Ellip,-}}{2} \Id
	\]
	is still positive definite.
	Denoting by $\tilde{E}_n^L(t)$ the eigenvalues of $\tilde{H}_t^L$,
	enumerated non-decreasingly and counting multiplicities, and by $\tilde{\lambda}_m^L(t)$
	the eigenvalues of $\tilde{H}_t^L$
	that are analytic, cf.~\cite[VII.Theorem 3.9]{Kato-80}, we have
	\begin{equation}
		\Phi_\ell(t)
		=
		\sum_{n\in\NN}\rho_\eps(\tilde{E}_n^L(t)-E-2\eps)
		=
		\sum_{m\in\NN}\rho_\eps(\tilde{\lambda}_m^L(t)-E-2\eps)
		.
		\label{eq:Phi}
	\end{equation}
	In light of the latter, $\Phi_\ell$ is differentiable as a composition of differentiable functions
	and it is legitimate to apply Lemma \ref{lem:PI-singular-distributions} with $\eta\in(a,b+\alpha)$, $b=m$ and $a,\alpha$ as above.
	Thereby, the inner integral in \eqref{eq:expectation-telescoping-sum} is bounded by
	\eqs{
		\int \bigl(\Phi_\ell(\omega_{r_\ell} +\eta) - &\Phi_\ell(\omega_{r_\ell} )\bigr)
		\Diff{\mu_{r_\ell}(\omega_{r_\ell})}
		\\
		&\leq
		\int \left(\Phi_\ell(\omega_{r_\ell} +4\eps / C_{\mathrm{subst}}) - \Phi_\ell(\omega_{r_\ell} )\right)\Diff{\mu_{r_\ell}(\omega_{r_\ell})}
		\\
		&\leq
		s(4\eps / C_{\mathrm{subst}}) \bigl( \Phi_\ell(m+\alpha+4\eps / C_{\mathrm{subst}}) - \Phi_\ell(-\alpha) \bigr)
		\\
		&\leq
		s(4\eps / C_{\mathrm{subst}})(-\Phi_\ell(-\alpha))
	.
	}
	Since $|\rho_\eps|\leq 1$ , we further estimate
	\eqs{
		-\Phi_\ell(-\alpha) &= \sum_{n\in\NN}(-\rho_\eps)(E_n^L(\omega^{\perp}+C_{\mathrm{subst}}^{-1}\cdot e_{\ell-1}^{(r)}-\alpha\cdot e_{r_\ell})-E-2\eps) \\
		&\leq \#\{ n \colon E_n^L(\omega^{\perp}+C_{\mathrm{subst}}^{-1}\cdot e_{\ell-1}^{(r)}-\alpha\cdot e_{r_\ell})\in\supp\rho_\eps(\cdot-E-2\eps) \}\\
		&\leq \#\{ n \colon  E_n^L(\omega^{\perp}+C_{\mathrm{subst}}^{-1}\cdot e_{\ell-1}^{(r)}-\alpha\cdot e_{r_\ell}) \leq E+3\eps\} \\
		&\leq \#\{ n \colon  E_n^L(\omega^{\perp}-\alpha\cdot e_{r_\ell}) \leq E_+\}
		.
	}
	Inserting $\alpha= E_-^{-2} \, \delta_-^{-N_{\unif}\cdot(1+E_+^{2/3})}$ in this estimate, 
	classical Weyl asymptotics, see, e.g., \cite[Corollary~4.1.26]{Stollmann-01}, imply 
	that 	there is a constant $C_{\mathrm{Weyl}}>0$, 
	depending only on the model parameters and $[E_-,E_+]$
	such that
	\[
		-\Phi_\ell(-\alpha)
		\leq
		\#\{  n \colon E_n^L(\omega^{\perp}-\alpha\cdot e_{r_\ell}) \leq E_+\}
		\leq
		C_{\mathrm{Weyl}} L^d
		,
	\]
	leading to the bound
	\be\label{eq:quadratic-in-volume}
		\int \Big( \Phi_\ell(\omega_{r_\ell} +C_{\mathrm{subst}}^{-1}) - \Phi_\ell(\omega_{r_\ell})\Big) \Diff{\mu_{r_\ell}(\omega_{r_\ell})}
		\leq
		C_{\mathrm{Weyl}} \cdot s(4\eps / C_{\mathrm{subst}}) L^d
		.
	\ee
	We bound the global modulus of continuity by
	$s(1/C_{\mathrm{subst}}) \leq \lceil 1/C_{\mathrm{subst}} \rceil s(\eps) $ using a simple
	subadditivity argument.
	Finally, this leaves us with
	\eqs{
		\EE\bigl(\Tr\bigl[\chi_{[E-\eps,E+\eps]}(H_\omega^L)\bigr]\bigr)
		&\leq
		C_{\mathrm{Weyl}} 	s(4\eps / C_{\mathrm{subst}}) L^d \# Q
		\\
		&\leq
		C_{\mathrm{Weyl}} \lceil 4/C_{\mathrm{subst}} \rceil s(\eps) (2+\delta_+)^d L^{2d}
		,
	}
	which proves \eqref{eq:Wegner-est} with $C_W:= C_{\mathrm{Weyl}} \lceil 4/C_{\mathrm{subst}} \rceil (2+\delta_+)^d$,
	depending indeed only on the model parameters and $[E_-,E_+]$.
\end{proof}

	\begin{remark}
		If we could prove \eqref{eq:quadratic-in-volume} with a right hand side independent from $L$, we would obtain a Wegner estimate that is linear in 
		the volume of the cube. 
		This would yield some information on the regularity of the integrated density of states. 
	\end{remark}

	\begin{remark}
		It would also be possible to give a proof of the Wegner estimate that is very similar to the proof given in \cite{Kirsch-96}, 
		see also \cite{KirschSS-98a,Stollmann-98,Veselic-08}.
		These proofs are rigorous variants of Wegner's original idea in \cite{Wegner-81}.
	\end{remark}
		
\section{Stronger results for small energies} \label{sec:small-energies}

\subsection{Unique continuation estimates for the gradient}
	As noted in Remark~\ref{rem:gradient} above, Lemma~\ref{lem:gradient} applies to discontinuous 
	matrix functions as well.
	In order to make use of this, we need to replace Theorem~\ref{thm:sfUCP-TV19} 
	by some appropriate version that holds true for such coefficient functions.
	Such versions are at disposal if we consider only energies near the minimum of the spectrum.
	In fact, such versions are available for Dirichlet as well as for 
	Neumann boundary conditions on $\partial\L$. 
	
\subsubsection{Dirichlet b.c.}
	In \cite{TautenhahnV-20}, the authors prove an uncertainty relation that implies an unique continuation estimate with a constant independent of the 
	Lipschitz-constant $\theta_\Lipschitz$, provided the considered energies are close to zero.
	In fact, the independence of the Lipschitz-constant is crucial, since it (as already noted in \cite{TautenhahnV-20}) allows us to combine the uncertainty 
	relation with an approximation argument to allow matrices that do not satisfy \eqref{eq:Lipschitz}. 
	Since the proof of the approximation argument was not spelled out in \cite{TautenhahnV-20} we provide it in Appendix~\ref{sec:appendix}.
	
	As a corollary of the latter we obtain the following result.

	\begin{corollary}\label{cor:low-energy-uncertainty}
		Let $L\in\NN$, let $A$ be a matrix function satisfying \eqref{eq:elliptic} and let $\delta\in (0,\delta_0/2)$. 
		Then there is a constant $\kappa'$ (given in \eqref{eq:kappa}) such that for all $(1,\delta)$-equidistributed sequences $Z=(z_j)_{j\in\ZZ^d}$, 
		all $\lambda< \kappa'$, and all $\psi\in \Ran \chi_{(-\infty, \lambda)}(H^L(A))$ we have
		\be
			\norm{\psi}_{L^2(S_{Z,\delta}(L))}^2\geq \kappa'\norm{\psi}_{L^2(\L_L)}^2.
			\label{eq:low-energy-uncertainty}
		\ee
	\end{corollary}
	
	We may now replace the use of Theorem~\ref{thm:sfUCP-TV19} in the proof of Theorem~\ref{thm:sfUCPG} with the last mentioned Corollary
	\ref{cor:low-energy-uncertainty} to obtain a version of our main result for small energies without assuming Lipschitz continuity of the matrix function $A$.
	
	\begin{theorem} \label{thm:sfUCPG-low-energy}
		Let $L\in\NN$, assume that $A$ satisfies \eqref{eq:elliptic} and let $\delta\in (0,\delta_0)$.
		Then there exists $\kappa>0$, depending on $\delta,\theta_{\Ellip,-}$ and the dimensions $d$, such that for all $0<E_-<E_+\leq\kappa$, 
		all $\psi\in\cD(H^L(A))$ satisfying $H^L(A)\psi=E\psi$ for some $E\in(E_-,E_+)$ and all $(1,\delta)$-equidistributed sequences $Z$
		\be
			\norm{\grad\psi}_{L^2(S_{\delta,Z}(L))}^2
			\geq \tilde{C}_{\sfUCP}^\grad\norm{\psi}_{L^2(\L_L)}^2.
			\label{eq:sfUCPG-low-energy}
		\ee
		holds true. 
		The constants $\tilde{C}_{\sfUCP}^\grad$ and $\kappa$ are given in 
		\eqref{eq:sfUCPG-low-energy-const} and \eqref{eq:sfUCPG-low-energy-energy-bnd}
		below. 
	\end{theorem}

	\begin{proof}
		Let $\kappa'$  be as in \eqref{eq:kappa} and choose
		\be
			\kappa=\kappa(\delta):=\kappa'(\delta/2)
			.
			\label{eq:sfUCPG-low-energy-energy-bnd}
		\ee
		An application of Lemma~\ref{lem:gradient} provides us with
		\[
			\norm{\grad\psi}_{L^2(S_{Z,\delta}(L))}^2 \geq C^\grad(\delta) \norm{\psi}_{L^2(S_{Z,\delta/2}(L))}^2
		\]
		and since $E_+\leq \kappa(\delta)=\kappa'(\delta/2)$ and $Z$ is
		also $(1,\delta/2)$-equidistributed, Corollary~\ref{cor:low-energy-uncertainty} shows
		\[
			\norm{\psi}_{L^2(S_{Z,\delta/2}(L))}^2 \geq \kappa'(\delta/2) \norm{\psi}_{L^2(\L_L)}^2.
		\]
		Combining these inequalities, we 
		obtain \eqref{eq:sfUCPG-low-energy} where the constant is given by
		\be
			 \tilde{C}_{\sfUCP}^\grad =  \frac{1}{2}\frac{\delta^2E_-^2}{2\theta_{\Ellip,+}\bigl(8\theta_{\Ellip,+}+\delta^2E_-\bigr)}\,
			 \Bigl(\frac{\delta}{2}\Bigr)^{M(1+\theta_{\Ellip,-}^{-2/3})},
			 \label{eq:sfUCPG-low-energy-const} 
		\ee
		where $M$ depends only on the dimension $d$.
	\end{proof}

	For energy intervals $(E_-,E_+)$ higher up in the spectrum it is unclear whether one can expect an estimate like \eqref{eq:sfUCPG-low-energy} to hold 
	without assuming Lipschitz continuity of the coefficients. 
	It is well known that there are operators with H\"older continuous coefficients which do not obey the (local) unique continuation principle,
	see \cite{Plis-60,Miller-73,Mandache-98}.

\subsubsection{Neumann b.c.} \label{ss:Neumann}

	We may compare Corollary~\ref{cor:low-energy-uncertainty} with the 
	recent result \cite{StollmannS-21} for divergence-type operators 
	with Neumann boundary conditions in dimensions $d\geq 3$.
	In order to formulate it, let $H^L_N(A)$ be the unique operator associated with the form $\fh^L$ given in \eqref{eq:eigenvalue} but with 
	the domain $\cD(\fh^L)=H^1(\L_L)$.

	\begin{theorem}[{special case of \cite[Theorem 1.1]{StollmannS-21}}]\label{thm:StollmanStolz-ur-low-energy}
		Let $L\in\NN_\infty$. 
		Assume that $d\geq 3$ and that the matrix function $A$ satisfies \eqref{eq:elliptic}.
		Then there are constants $C,a,b,c>0$ depending only on the dimension $d$, such that for all $\delta\in (0,1/2)$, all $(1,\delta)$-equidistributed 
		sequences $Z$, and all $\psi\in \chi_I(H^L_N(A))$ 
		(where $I=[0,C\theta_{\Ellip,-}\delta^{d-2}]$) we have
		\be
			\norm{\psi}_{L^2(S_{Z,\delta}(L))}^2 \geq C^N_\sfUCP \norm{\psi}_{L^2(\L_L)}^2.
			\label{eq:StollmanStolz-ur-low-energy}
		\ee
		The constant $C^N_\sfUCP$ is explicitly given by
		\be
			C^N_\sfUCP=C^N_\sfUCP(\delta)=
			c\theta_{\Ellip,-}\delta^d\Bigl[\frac{b}{(\min\{\sqrt{d},L/2\})^2}
			+\bigl|\log\bigl(a\delta^{d-2}\bigr)\bigr|\Bigr]^{-2}.
			\label{eq:StollmanStolz-ur-low-energy-const}
		\ee
	\end{theorem}

	It is worth emphasizing that Theorem~\ref{thm:StollmanStolz-ur-low-energy}, as Corollary~\ref{cor:low-energy-uncertainty} above, does not need the Lipschitz condition
	\eqref{eq:Lipschitz} nor the assumption \eqref{eq:DIR-assumption}.

	\begin{remark}
		Actually, the result of \cite{StollmannS-21} only uses the lower ellipticity constant $\theta_{\Ellip,-}$ of $A$.
		Since $A$ is assumed to satisfy \eqref{eq:elliptic} 
		in our application, we do not elaborate further on this detail.
		It should also be mentioned that the main result of \cite{StollmannS-21} is applicable in more general situations.
	\end{remark}
	
	With Theorem~\ref{thm:StollmanStolz-ur-low-energy} at hand, it is possible to prove a unique continuation estimate for the gradient of eigenfunctions
	of divergence-type operators with Neumann boundary conditions.

	\begin{theorem} \label{thm:sfUCPG-low-energy-Neumann}
		Let $L\in\NN_\infty$, $d\geq 3$, and suppose that $A$ satisfies \eqref{eq:elliptic}. 
		Let $\delta\in (0,1/2)$ and let $Z=(z_j)_{j\in\ZZ^d}$ be a $(1,\delta)$-equidistributed sequence.
		Then there exists $\kappa^N>0$, depending on $\delta,\theta_{\Ellip,-}$, and the dimension $d$, such that
		for all $0<E_-<E_+\leq\kappa^N$ the following holds: 
		There is a constant $C_{\sfUCP}^{N,\grad}>0$ such that for all $\psi\in\cD(H^L_N(A))$ satisfying $H^L_N(A)\psi=E\psi$ for some $E\in(E_-,E_+)$
		we have
		\be
			\norm{\grad\psi}_{L^2(S_{\delta,Z}(L))}^2
			\geq C_{\sfUCP}^{N,\grad}\norm{\psi}_{L^2(\L_L)}^2.
			\label{eq:sfUCPG-low-energy-Neumann}
		\ee
		Here the constants are given by
		\be
			C_{\sfUCP}^{N,\grad} = C^\grad(\delta) C^N_\sfUCP(\delta/2)
			\quad\text{and}\quad\kappa^N=\kappa^N(\delta)=C\theta_{\Ellip,-}\Bigl(\frac{\delta}{2}\Bigr)^{d-2},
			\label{eq:sfUCPG-low-energy-Neumann-const}
		\ee
		where $C$ depends only on the dimension $d$.
	\end{theorem}
	
	\begin{proof}
		The proof is an easy adaptation of the proof of Theorem~\ref{thm:sfUCPG-low-energy}.
	\end{proof}

	In contrast to Theorem~\ref{thm:sfUCPG-low-energy}, we do not need to 
	assume that $L$ is finite. 
	This assumption is needed in the last mentioned theorem 
	since it relies on an approximation argument that uses that the 
	limit operator has purely discrete spectrum.

	\begin{remark}
		In the case of Neumann b.c.~it is trivial to see that 
		unique continuation for the gradient fails at $0$:
		In fact, for all $L\in\NN$ the constant function $\psi\equiv 1$ 
		is an eigenfunction of $H^L_N(A)$ corresponding to the eigenvalue $0\in\sigma(H^L_N(A))$. 
		Since $\grad\psi \equiv  0$ on the whole cube $\L_L$, unique 
		continuation for the gradient cannot hold for this eigenfunction.
	\end{remark}

\subsection{Eigenvalue lifting and Wegner estimates at low energies}

	The unique continuation estimates for the gradient of eigenfunctions 
	corresponding to eigenvalues close to zero stated in the previous subsection
	allow us
	to prove some of our results from 
	Section~\ref{sec:applications} for more general models.

	To begin with, replacing Theorem~\ref{thm:sfUCPG} in the proof of Theorem~\ref{thm:evl} by Theorem~\ref{thm:sfUCPG-low-energy}, allows us to prove the
	following variant of our eigenvalue lifting estimate for energies close to zero. 
	Here we do not need the assumption \eqref{eq:Lipschitz} for the matrix function $A$.

	\begin{theorem} \label{thm:evl-low-energy}
		Let $T>0$, $\delta\in (0,1/2)$, assume that $A$ satisfies \eqref{eq:elliptic}, let $\kappa$ be the constant from Theorem~\ref{thm:sfUCPG-low-energy}
		and let $0<E_-<E_+<\kappa$.
		Then there is a constant $\tilde{C}_\evl>0$, such that for all $(1,\delta)$-equidistributed sequences $Z$, all $W\in L^\infty(\L_L)$ satisfying 
		$W\geq \indic_{S_{Z,\delta}(L)}$ and all $n\in\NN$ such that $E_-<E_n^L(0)\leq E_n^L(T)<E_+$ we have
		\be
			E_n^L(t) \geq E_n^L(0)+t\,\tilde{C}_\evl.
			\label{eq:evl-low-energy}
		\ee
		The constant is given by
		\[
			\tilde{C}_\evl=\frac{\delta^2E_-^2}{4\theta_{\Ellip,+}\bigl(8\theta_{\Ellip,+}+\delta^2E_-\bigr)}
			\Bigl(\frac{\delta}{2}\Bigr)^{M(1+\theta_{\Ellip,-}^{-2/3})}.
		\]
		where $M$ depends only on the dimension $d$.
	\end{theorem}

	An application of Theorem~\ref{thm:sfUCPG-low-energy-Neumann} (that is based on \cite{StollmannS-21})
	provides us with the same result for Neumann boundary conditions 
	and energies close to zero. 
	We again consider the forms defined in \eqref{eq:evl-form} but with 
	domain given by $\cD(\fh_t^L)=H^1(\L_L)$.
	Moreover, let $H^L_{N,t}:=H^L_N(A+t\,W\Id)$.
	Then with the same arguments as above we obtain the following eigenvalue lifting.

	\begin{theorem}
		Let $d\geq 3$, $T>0$, $\delta\in (0,1/2)$, $Z$ be a $(1,\delta)$-equidistributed sequence, assume that $A$ satisfies
		\eqref{eq:elliptic}, and let $\kappa^N$ be as in \eqref{eq:sfUCPG-low-energy-Neumann-const}. 
		Moreover, let $W\in L^\infty(\L_L)$ satisfy $W\geq \indic_{S_{Z,\delta}(L)}$, and let $0<E_-<E_+\leq\kappa^N$.
		Then for all $n\in\NN$ such that $E_-<E_n^L(0)\leq E_n^L(T)<E_+$ the eigenvalues of the divergence-type operator $H^L_{N,t}$ with Neumann boundary
		conditions obey
		\be
			E_n^L(t) \geq E_n^L(0)+t\,C^N_\evl,
			\label{eq:evl-low-energy-N}
		\ee
		where
		\[
			C^N_\evl=C^{N,\grad}_\sfUCP(\delta)
			=cC^\grad(\delta)\theta_{\Ellip,-}\delta^d
			\Bigl[\frac{b}{(\min\{\sqrt{n},L/2\})^2}
			+\bigl|\log\bigl(a\delta^{d-2}\bigr)\bigr| \Bigr]^{-2}
		\]
		with constants $a,b,c>0$ depending only on the dimension $d$.
	\end{theorem}
	
	The new eigenvalue-lifting results allows us to 
	formulate more general Wegner estimates for small energies:
	In our Model~\ref{model:A} we required the matrix function
	$A$ and the single-site perturbations $u_j$ to be Lipschitz continuous. 
	However, since Theorem~\ref{thm:sfUCPG-low-energy} does not require 
	Lipschitz continuity of the coefficients
	it is also not necessary to require this property for the 
	Wegner estimate \emph{for small energies}. 
	Hence, we are in the position to consider the following model. 

	\begin{mmodel}{(B)} \label{model:B}
		A model of type~\ref{model:A} where we do \emph{not} assume 
		\begin{enumerate}[(i)]
			\item that the matrix function $A$ satisfies \eqref{eq:Lipschitz},
			\item that the single-site perturbations $u_j$ are 
			Lipschitz continuous, 
			\item that the matrix function $A$ satisfies \eqref{eq:DIR-assumption}.
		\end{enumerate}
	\end{mmodel}

	For an operator $H^L_\omega:=H^L(A_\omega)$ of type~\ref{model:B}
	we obtain the following Wegner estimate. 

	\begin{theorem}
		Consider a model of type~\ref{model:B}.
		Let $\kappa$ be as in Theorem~\ref{thm:sfUCPG-low-energy}. 
		Then for all $0<E_-<E_+<\kappa$ there exists a constant $C_W>0$ such that for all $L\in \NN$, $E\in\RR$ and $\eps>0$ with 
		$[E-7\eps,E+7\eps]\subset[E_-,E_+]$ we have  
		\[
			\EE\Bigl(\Tr\chi_{[E-\eps,E+\eps]}\bigl(H_\omega^L\bigr)\Bigr) \leq C_W\eps|\L_L|^2.
		\] 
	\end{theorem}

	The Wegner estimate for the random divergence-type operator
	 $H^L_{N,\omega}:=H^L_N(A_\omega)$ with Neumann boundary conditions 
	 proceeds completly analogous.

	\begin{theorem}
		Consider a model of type~\ref{model:B}.
		Let $\kappa^N$ be as in Corollary~\ref{thm:sfUCPG-low-energy-Neumann}. 
		Then for all $0<E_-<E_+\leq\kappa^N$ there exists a constant $C_W>0$ such that for all $L\in \NN_\infty$, $E\in\RR$ and $\eps>0$ with 
		$[E-7\eps,E+7\eps]\subset[E_-,E_+]$ we have
		\[
			\EE\Bigl[\Tr\chi_{[E-\eps,E+\eps]}\bigl(H_{N,\omega}^L\bigr)  \Bigr] \leq C_W\eps|\L_L|^2.
		\]
	\end{theorem}

\section{Scaling}\label{sec:scaling}

	Further applications we have in mind require scaled variants of our main results. 
	In particular, we are interested in scaled variants of Theorem~\ref{thm:sfUCPG} 
	and variants of the eigenvalue lifting spelled out in Theorem~\ref{thm:evl}
	and Corollary~\ref{cor:evl-discont-perturbation}.

	Let $G>0$ and consider the cube $\L_{GL}=G\L_L$. 
	We define the scaling $S\colon \RR^d\to\RR^d$ given by $S(x):=Gx$ and set $h_G=h\circ S$ for all functions $h$ defined on $\L_{GL}$.
	Using the definition of the divergence-type operators via forms, 
	it is easy to see that every eigenfunction $\psi\in L^2(G\L_L)$ corresponding
	to some eigenvalue $E\in\sigma(H^{GL}(A))$ satisfies $H^L_G(A)\psi_G=G^2E\psi_G$, where $H^L_G(A_G)$ is the operator associated to the form
	\[
		\fh^L_G\colon H^1_0(\L_L)\times H^1_0(\L_L) \to \CC, \quad (u,v)\mapsto \int_{\L_L} \overline{\grad u}\cdot A_G \grad v.
	\]
	Note that $A_G$ satisfies \eqref{eq:elliptic} with the ellipticity constants $\theta^G_{\Ellip,\pm}=\theta_{\Ellip,\pm}$, \eqref{eq:Lipschitz} with Lipschitz-constant 
	$\theta_{\Lipschitz}^G=G\theta_\Lipschitz$, and if $A$ satisfies \eqref{eq:DIR-assumption} on the cube $\L_{GL}$ then $A_G$ satisfies \eqref{eq:DIR-assumption} on the 
	cube $\L_L$.
	For some $\delta\in(0,G/2)$ we let $Z=(z_j)_{j\in(G\ZZ)^d}$ be a 
	$(G,\delta)$-equi\-distri\-buted sequence and calculate
	\bes
		\int_{S_{Z,\delta}(GL)} |\grad u|^2 
		= G^d\int_{G^{-1}S_{Z,\delta}(GL)}|(\grad u)_G|^2 
		= G^{d-2}\int_{G^{-1}S_{Z,\delta}(GL)}|\grad u_G|^2
	\ees
	for all $u\in H^1(\L_{GL})$. 
	Since $G^{-1}S_{Z,\delta}(GL)=S_{Z_G,\delta/G}(L)$ for some $(1,\delta/G)$-equi\-distri\-buted sequence $Z_G$, we are in the position to apply our main results. 
	This way we prove the next two Corollaries.
	
	\begin{remark}
		The results in this section are only stated for Dirichlet boundary conditions. 
		However, as seen in the previous sections, it is possible to treat Neumann boundary conditions with similar arguments.
	\end{remark}
	
	\begin{corollary} \label{cor:scaled-sfUCPG}
		Let $L\in\NN_\infty$. 
		Assume that $A$ satisfies \eqref{eq:elliptic}, \eqref{eq:Lipschitz} and \eqref{eq:DIR-assumption} on the cube $\L_{GL}$. 
		Let $0<E_-<E_+<\infty$ and let $\delta_0$ be sufficiently small, depending on $d,\theta_{\Ellip},\theta_\Lipschitz, G$. 
		Then for all $\delta\in (0,\delta_0)$ there exists a constant $C_{\sfUCP,G}^\grad>0$, depending on $d,\theta_{\Ellip,\pm},G\theta_\Lipschitz,E_-,E_+$ 
		and $\delta/G$ such that for all $E_-<E<E_+$, all $\psi\in\cD(H^{GL}(A))$ satisfying $H^{GL}(A)\psi = E\psi$,
		and all $(G,\delta)$-equidistributed sequences we have
		\be
			\norm{\grad\psi}_{L^2(S_{Z,\delta}(GL))}^2 \geq C_{\sfUCP,G}^\grad \norm{\psi}_{L^2(G\L_L)}^2.
			\label{eq:scaled-sfUCPG}
		\ee
		The constant is given by
		\[
			C^\grad_{\sfUCP,G}=C^\grad_{\sfUCP,G}(\delta)
			=\frac{\delta^2E_-^2}{2\theta_{\Ellip,+}\bigl(8\theta_{\Ellip,+}+\delta^2E_-\bigr)}
			\Bigl(\frac{\delta}{2G}\Bigr)^{N(1+G^{4/3}E_+^{2/3})},
		\]
		where $N=N(\delta,\theta_\Ellip,G\theta_\Lipschitz)$ is the constant from Theorem~\ref{thm:sfUCP-TV19} with $\theta_\Lipschitz$ replaced by $G\theta_\Lipschitz$.
	\end{corollary}

	\begin{remark}
		The scaling procedure described above also provides an appropriate choice for $\delta_0$, namely
		$\delta_0 = 2G\bigl(330d\e^2\theta_\Ellip^{11/2}(\theta_\Ellip+1)^{5/3}(G\theta_\Lipschitz+1)\bigr)^{-1}$.
	\end{remark}
	
	\begin{corollary}\label{cor:scaled-sfUCPG-low-energy}
		Let $L\in\NN$ and $G>0$. 
		Assume that  $A$ satisfies \eqref{eq:elliptic} on the cube $\L_{GL}$ and let $\delta\in(0,\delta_0)$.
		Then there are a constants $\kappa_G>0$ and $\tilde{C}^\grad_{\sfUCP,G}>0$, 
		depending on $\delta,\theta_{\Ellip,-},G$ and the dimension $d$, 
		such that for all $0<E_-<E_+<\kappa_G$, all $\psi\in \cD(H^{GL}(A))$ satisfying $H^{GL}(A)\psi=E\psi$ for some $E\in (E_-,E_+)$ and all 
		$(G,\delta)$-equidistributed sequences $Z$ we have
		\be
			\norm{\grad\psi}_{L^2(S_{\delta,Z}(GL)}^2 
			\geq \tilde{C}^\grad_{\sfUCP,G} \norm{\psi}_{L^2(\L_{GL})}^2
			.
			\label{eq:scaled-sfUCPG-low-energy}
		\ee 
		The constants are given by
		\be
			\kappa_G=\frac{1}{2G^2}\Bigl(\frac{\delta}{2G}\Bigr)^{M(1+\theta_{\Ellip,-}^{-2/3})}
			\label{eq:scaled-sfUCPG-low-energy-const-1}
		\ee
		and
		\be
			\tilde{C}^\grad_{\sfUCP,G} = \frac{\delta^2E_-^2}{2\theta_{\Ellip,+}\bigl(8\theta_{\Ellip,+}+\delta^2E_-\bigr)}
			\Bigl(\frac{\delta}{2G}\Bigr)^{M(1+\theta_{\Ellip,-}^{-2/3})},
			\label{eq:scaled-sfUCPG-low-energy-const-2}
		\ee
		where $M$ is a constant that depends only on the dimension.
	\end{corollary}
	
	There are also scaled variants of the results on eigenvalue lifting, 
	Theorem~\ref{thm:evl} and Corollary~\ref{cor:evl-discont-perturbation}. 
	In order to formulate these, we denote by $E_n^{GL}(t)$ the
	eigenvalues of the operator $H^{GL}_t:=H^{GL}(A+t\,W\Id)$ enumerated non-decreasingly and counting multiplicities.

	\begin{corollary} \label{cor:scaled-evl}
		Let $T,\kappa,K_1,K_2,G>0, \delta\in (0,1/2)$ and let $0<E_-<E_+<\infty$. 
		Assume that $A$ satisfies \eqref{eq:elliptic}, \eqref{eq:Lipschitz}
		and \eqref{eq:DIR-assumption} on the cube $\L_{GL}$. 
		Then there exists a constant $C_{\evl,G}>0$ such that for all 
		$(G,\delta)$-equidistributed sequence $Z$, all
		Lipschitz continuous $W\in L^\infty(\L_{GL})$ 
		satisfying $\Lipschitz(W) \geq K_1$, $\norm{W}_\infty \leq K_2$ 
		and $W\geq \kappa\indic_{S_{Z,\delta}(GL)}$, and all $n\in\NN$ 
		such that $E_-<E_n^{GL}(0)\leq E_n^{GL}(T)<E_+$ we have
		\be
			E_n^{GL}(t)\geq E_n^{GL}(0) + \kappa t\,C_\evl^G.
			\label{eq:scaled-evl}
		\ee
		The constant is given by
		\be
			C_{\evl,G}=\frac{\delta^2E_-^2}{2\theta_{\Ellip,+}'(8\theta_{\Ellip,+}'+\delta^2E_-)}
			\Bigl(\frac{\delta}{2G}\Bigr)^{N(1+G^{4/3}E_+^{2/3})}.
			\label{eq:scaled-evl-const}
		\ee
		Here $\theta'_{\Ellip,-}=\theta_{\Ellip,-}$,
		$\theta'_{\Ellip,+}=\theta_{\Ellip,+}+TK_2,
		\theta'_\Lipschitz = \theta_\Lipschitz+TK_1$ and 
		$N=N(\delta,\theta'_\Ellip,G\theta'_\Lipschitz)$.
	\end{corollary}

	\begin{corollary} \label{cor:scaled-evl-discont-perturbation}
		Let $T,G>0, \delta\in (0,1/2)$ and let $0<E_-<E_+<\infty$. 
		Assume that $A$ satisfies \eqref{eq:elliptic}, \eqref{eq:Lipschitz} 
		and \eqref{eq:DIR-assumption} on the cube $\L_{GL}$. 
		Then, there exists a constant $\hat{C}_{\evl,G}$, 
		such that for all $(G,\delta)$-equidistributed sequences $Z$, all 
		$W\in L^\infty(\L_{GL})$ satisfying 
		$W\geq \indic_{S_{Z,\delta}(GL)}$, and all $n\in\NN$ 
		such that $E_-<E_n^{GL}(0)\leq E_n^{GL}(T)<E_+$ we have
		\be
			E_n^{GL}(t)\geq E_n^{GL}(0) + t\,\hat{C}_{\evl,G}.
			\label{eq:scaled-evl-discont-perturbation}
		\ee
	\end{corollary}
	
	In addition, we may formulate the scaled variant of Theorem~\ref{thm:evl-low-energy}.
	
	\begin{corollary}\label{thm:scaled-evl-low-energy}
		Let $T,G>0, \delta\in (0,\delta_0)$ and assume that $A$ satisfies \eqref{eq:elliptic} on the cube $\L_{GL}$, let $\kappa_G$ be the constant from
		Corollary~\ref{cor:scaled-sfUCPG-low-energy} and let $0<E_-<E_+\leq \kappa_G$. 
		Then there exists a constant $\tilde{C}_{\evl,G}$, such that for all $(G,\delta)$-equidistributed sequences $Z$ and all $W\in L^\infty(\L_{GL})$ 
		satisfying $W\geq\indic_{S_{Z,\delta}(GL)}$ we have
		\be
			E_n^{GL}(t)\geq E_n^{GL}(0)+t\,\tilde{C}_{\evl,G}
			\label{eq:scaled-evl-low-energy}
		\ee
		for all $n\in\NN$ such that $E_-<E_n^{GL}(0)\leq E_n^{GL}(T)<E_+$. 
		The constant is given by
		\[
			\tilde{C}_{\evl,G}=\frac{\delta^2E_-^2}{2\theta_{\Ellip,+}\bigl(8\theta_{\Ellip,+}+\delta^2E_-\bigr)}
			\Bigl(\frac{\delta}{2G}\Bigr)^{M(1+\theta_{\Ellip,-}^{-2/3})}.
		\]
	\end{corollary}
	
\appendix

\section[Proof of a remark from Tautenhahn-Veseli\'c]{Proof of a remark from \cite{TautenhahnV-20}} \label{sec:appendix}

	Let us first cite a special case of Theorem~3.8 in \cite{TautenhahnV-20} that eliminates the dependence on the constant $\theta_\Lipschitz$ and the 
	condition \eqref{eq:DIR-assumption}.

	\begin{theorem}[{cf.~\cite[Theorem~3.8]{TautenhahnV-20}}] \label{thm:ur-low-energy}
		Assume that the matrix function $A$ satisfies \eqref{eq:Lipschitz} and \eqref{eq:elliptic}.
		Then for all $\delta\in (0,\delta_0/2)$ and all $(1,\delta)$-equidistributed sequences $Z$ the uncertainty relation
		\be
			\chi_I(H^L(A))\indic_{S_{Z,\delta}(L)}\chi_I(H^L(A))\geq \kappa'\,\chi_I(H^L(A))
			\label{eq:ur-low-energy}
		\ee
		holds for all measurable $I\subset(-\infty,\kappa']$. 
		Here
		\be
			\kappa'=\kappa'(\delta) 
			= \delta^{M(1+\theta_{\Ellip,-}^{-2/3})} / 2
			\label{eq:kappa}
		\ee
		with some constant $M$ depending only on the dimension.
	\end{theorem}

	\begin{remark}
		The uncertainty relation \eqref{eq:ur-low-energy} implies that
		\[
			\norm{\psi}_{L^2(S_{Z,\delta}(L))}^2 \geq \kappa' \norm{\psi}_{L^2(\Lambda_L)}^2
		\]
		for an eigenfunction $\psi$ of $H^L(A)$ corresponding to an eigenvalue $E\leq \kappa'$.
	\end{remark}

	Now we assume that $A\colon \L_L\to\RR^{d\times d}$ satisfies \eqref{eq:elliptic}. 
	In order to approximate the operator $H^L:=H^L(A)$ by a sequence of operators 
	$(H^L_\ell)_\ell$ satisfying the assumption of 
	Theorem~\ref{thm:ur-low-energy} we need to approximate $A$ 
	by a sequence of Lipschitz continuous and uniformly elliptic matrix 
	functions $(A_\ell)_\ell$.

	\begin{lemma}\label{lem:Lipschitz-approx-mat}
		Let $\emptyset\neq U\subset\RR^d$ be bounded open set,
		let $A=(a_{j,k})_{j,k=1}^d\colon U\to\Sym(d,\RR)$ be a matrix-function that 
		satisfies \eqref{eq:elliptic}, and let $0<\eps<\theta_{\Ellip,-}$.
		Then there exists a sequence of symmetric, uniformly elliptic, and Lipschitz continuous matrices $(A_\ell)_\ell$ with ellipticity constants 
		$\theta_{\Ellip,-}(A_\ell)=\theta_{\Ellip,-}-\eps$ and $\theta_{\Ellip,+}(A_\ell)=\theta_{\Ellip,+}$, converging to $A$ pointwise almost everywhere.
	\end{lemma}

	\begin{remark}
		Note that by the polarization identity the ellipticity of $A$ implies that $a_{j,k}\in L^\infty(U)\subset L^1(U)$ for all $j,k\in\{1,\dots,d\}$. 
	\end{remark} 

	\begin{proof}
		We consider $A$ as a matrix-function $\RR^d\to\Sym(d,\RR)$ by setting $A\equiv 0$ on $\RR^d\setminus U$.
		Let $B:=(\theta_{\Ellip,-}-\eps)\Id$ and $\tilde{A}=A-B$. 
		Then $\tilde{A}$ is elliptic with ellipticity constants $\theta_{\Ellip,-}(\tilde{A})=\eps$ and 
		$\theta_{\Ellip,+}(\tilde{A})=\theta_{\Ellip,+}-\theta_{\Ellip,-}+\eps$.
		Let $g\in C_c^\infty(\RR)$ satisfy $\supp g=[-1,1]$, $g\geq 0$ and $\int_\RR g=1$. 
		Moreover, suppose that $\norm{g}_\infty\leq M$ for some $M>0$. 
		We define $\phi\colon \RR^d\to\RR$ by $\phi(x):=g(|x|)$ and $\phi_\ell\colon \RR^d\to\RR$ by $\phi_\ell(x):=\ell^d\phi(\ell x)$ for $\ell\in\NN$.
		Then $\norm{\grad\phi_\ell}_\infty\leq \ell^{d+1}M$ and by Young's inequality for convolutions it is easy to see that 
		$\norm{\tilde{a}_{j,k}\ast\phi_\ell}_\infty\leq \norm{\tilde{a}_{j,k}}_\infty$.
		Furthermore, each $a_{j,k}\ast\phi_\ell$ is Lipschitz continuous with Lipschitz-constant at most $\ell^{d+1}M\max_{j,k}\norm{a_{j,k}}_1$.

		We define
		\[
			\tilde{A}_\ell=\tilde{A}\ast\phi_\ell:=(\tilde{a}_{j,k}\ast\phi_\ell)_{j,k=1\dots,d}.
		\]
		Then $\tilde{A}_\ell$ converges pointwise almost everywhere to $\tilde{A}$ and $\tilde{A}_\ell$ is Lipschitz continuous for all $\ell\in\NN$. 
		Using
		\[
			\xi\cdot\tilde{A}_\ell\xi = \sum_{j,k=1}^d (\tilde{a}_{j,k}\ast\phi_\ell)\xi_j\xi_k
			=\Bigl( \sum_{j,k=1}^d \tilde{a}_{j,k}\xi_j\xi_k  \Bigr)\ast\phi_\ell
		\]
		for all $\xi\in\RR^d$
		we obtain $0\leq\xi\cdot\tilde{A}_\ell\xi\leq \theta_{\Ellip,+}(\tilde{A})$.
		Thus, the approximation $A_\ell:=B+\tilde{A}_\ell$ is Lipschitz continuous and uniformly elliptic with 
		$\theta_{\Ellip,-}(A_\ell)= \theta_{\Ellip,-}-\eps$ and $\theta_{\Ellip,+}(A_\ell)= \theta_{\Ellip,+}$.
	\end{proof}

	Corresponding to a sequence of matrices as in Lemma~\ref{lem:Lipschitz-approx-mat} we define forms
	\[
		\fh^L_\ell(u,v) = \int_{\Lambda_L}~\overline{\grad u}\cdot A_\ell\grad v
	\]
	on $\cD(\fh^L_\ell)=H^1_0(\Lambda_L)$. 
	Each form $\fh^L_\ell$ generates a unique selfadjoint operator 
	$H^L_\ell=H^L(A_\ell)$ and since all the matrices $A_\ell$ are Lipschitz continuous,
	Proposition~\ref{thm:ur-low-energy} applies to $H_\ell^L$ for all $\ell\in\NN$.
	We aim to prove that the sequence of operators $(H_\ell^L)_\ell$ converges to $H^L$ in some appropriate sense, such that we can conclude
	an inequality like \eqref{eq:ur-low-energy} for $H^L$.

	For the sake of completeness we recap the the notion of $\Gamma$-convergence.
	\begin{definition}
		Let $F$ be a topological vector space, $f\in F$, and denote by $\mathcal{U}(f)$ the system of all open neighbourhoods of $f$ in $F$. 
		Then a sequence $(x_n)_{n\in\NN}$ of functions $x_n\colon F\to\RR$ is 
		said to $\Gamma$-converge to a function $x\colon F\to\RR$ if
		\[
			x(f)= \sup_{U\in\mathcal{U}(f)}\liminf_{n\to\infty} \inf_{g\in U} x_n(g)
			= \sup_{U\in\mathcal{U}(f)}\limsup_{n\to\infty} \inf_{g\in U} x_n(g)
		\]
		for all $f\in F$.
	\end{definition}
	
	The next proposition shows how the $\Gamma$-convergence of the forms $(\fh^L_\ell)$ implies that the sequence $(H_\ell)_\ell$ converges to $H$ in the strong resolvent sense.

	\begin{proposition} \label{prop:G-conv.-form-src-operator}
		The sequence of quadratic forms $(\fh^L_\ell)$ $\Gamma$-converges to $\fh^L$ in the weak topology of $H^1_0(\Lambda_L)$. 
		This implies that the corresponding operators $H_\ell^L$ converge to $H^L$ in the strong resolvent sense.
	\end{proposition}

	\begin{proof}
		The pointwise convergence
		\[
			\lim_{\ell\to\infty}\xi\cdot A_\ell(x)\xi = \xi\cdot A(x)\xi
		\]
		for every $\xi\in\RR^d$ and almost every $x\in\L_L$ 
		combined with the uniform ellipticity of the matrices $A_\ell$ imply 
		that $\fh_\ell^L$ $\Gamma$-converges to $\fh^L$ in the weak topology 
		of $H^1_0(\Lambda_L)$, see \cite[Theorem~5.14]{DalMaso-93}.
		Moreover \cite[Theorem~13.12 (cf. also Example~13.13)]{DalMaso-93} implies that in this case $H_\ell^L$ converges to $H^L$ in the strong resolvent sense.
	\end{proof}

	\begin{remark}
		If it is possible to approximate $A$ in a (in some sense) monotone way, one could replace the use of $\Gamma$-convergence by simpler or more elementary 
		arguments, cf.~\cite{Weidmann-84, Hundertmark-96}.
	\end{remark}

	Combining these results, we prove the next theorem, 
	which was the goal of this appendix.
	
	\begin{theorem} \label{thm:ur-low-energy-no-lipschitz}
		Let $L\in\NN$, let $A$ be a matrix function satisfying \eqref{eq:elliptic}, and let $\delta\in (0,\delta_0)$. With $\kappa'$ as in Theorem 
		\ref{thm:ur-low-energy}, for all $(1,\delta)$-equidistributed sequences $Z=(z_j)_{j\in\ZZ^d}$ and all intervals 
		$I:=(-\infty,\lambda)\subset (-\infty,\kappa']$ we have
		\be
			\chi_{I}(H^L)\indic_{S_{Z,\delta}(L)}\chi_{I}(H^L)
			\geq \kappa'\,\chi_{I}(H^L).
			\label{eq:ur-low-energy-no-lipschitz}
		\ee
	\end{theorem}
	
	\begin{proof}
		Since $H^L$ has purely discrete spectrum, there exists 
		$\tilde{\lambda}\in [\lambda,\kappa']\setminus\sigma(H^L)$.
		Set $\cJ=(-\infty,\tilde{\lambda}]$.
		Since $A_\ell$ is Lipschitz continuous, Theorem~\ref{thm:ur-low-energy} applies to $H_\ell^L$ and we obtain
		\be
			\chi_{\cJ}(H_\ell^L)\indic_{S_{Z,\delta}(L)}\chi_{\cJ}(H_\ell^L)
			\geq \kappa'\,\chi_{\cJ}(H_\ell^L)
			\label{eq:uncertainty-relation-low-energy-qf-k}
		\ee
		for all $\ell\in\NN$. 
		By \cite[Theorem~VIII.24]{ReedS-80} the convergence $H_\ell^L\to H^L$ 
		in the strong resolvent sense, $\|\chi_{\cJ}(H_\ell^L)\|\leq 1$, 
		and $\chi_{\{ \tilde{\lambda}\}}(H^L)=0$ imply that
		\bes
			 \lim_{\ell\to\infty}\chi_{\cJ}(H_\ell^L)\psi 
			 = \chi_{\cJ}(H^L)\psi \quad\text{for all}\quad \psi\in L^2(\L_L)
			 .
		\ees
		Clearly, the strong convergence of $\chi_{I}(H_\ell^L)$ implies
		\[
			\chi_{\cJ}(H_\ell^L)\indic_{S_{Z,\delta}(L)}\chi_{\cJ}(H_\ell^L)
			\to \chi_{I}(H^L)\indic_{S_{Z,\delta}(L)}\chi_{\cJ}(H^L)
		\]
		strongly as $\ell\to\infty$. 
		Multiplying both sides by $\chi_I(H^L)$, this, together with the observation that 
		the right hand side of \eqref{eq:uncertainty-relation-low-energy-qf-k} converges 
		strongly to $\kappa'\,\chi_{\cJ}(H^L)$, implies \eqref{eq:ur-low-energy-no-lipschitz}.
	\end{proof}

	\textbf{Acknowledgments.} The results presented in this paper were obtained 
	while the first author was employed at TU Dortmund University and 
	partially supported by the DFG grants VE 253/9-1 \emph{Random Schrödinger 
	operators with non-linear influence of randomness} 
	and VE 253/10-1 \emph{Quantitative unique continuation
	properties of elliptic PDEs with variable 2nd order coefficients and applications in
	control theory, Anderson localization, and photonics}.
%
%
%
\newcommand{\etalchar}[1]{$^{#1}$}

\end{document}